\documentclass[11pt,a4paper]{article}
\usepackage{fullpage,color}
\usepackage[mac]{inputenc}
\usepackage{amsmath,amsfonts,amssymb,amsthm,amscd}
\usepackage{graphicx}
\usepackage{verbatim}

\usepackage{hyperref}
\usepackage{mathtools}
\usepackage{tikz}
\usepackage{cleveref}
\usepackage{authblk}
\usepackage{siunitx}

\newtheorem{theorem}{Theorem}[section]

\newtheorem{remark}[theorem]{Remark}
\newtheorem{lemma}[theorem]{Lemma}
\newtheorem{corollary}[theorem]{Corollary}

\title{Explicit results for Euler's factorial series in arithmetic progressions under GRH}
\author{Neea Paloj\"arvi\thanks{The work was supported by the Emil Aaltonen Foundation.}}

\affil{Department of Mathematics and Statistics, University of Helsinki, P.O.\ Box 68, 00014 Helsinki}
\affil{neea.palojarvi@helsinki.fi}

\date{}

\begin{document}
\maketitle

\allowdisplaybreaks
\begin{abstract} 
In this article, we study the Euler's factorial series $F_p(t)=\sum_{n=0}^\infty n!t^n$ in $p$-adic domain under the Generalized Riemann Hypothesis. First, we show that if we consider primes in $k\varphi(m)/(k+1)$ residue classes in the reduced residue system modulo $m$, then under certain explicit
extra conditions we must have $\lambda_0+\lambda_1F_p(\alpha_1)+\ldots+\lambda_kF_p(\alpha_k) \neq 0$ for at least one such prime. We also prove an explicit $p$-adic lower bound for the previous linear form. Secondly, we consider the case where we take primes in arithmetic progressions from more than $k\varphi(m)/(k+1)$ residue classes. Then there is an infinite collection of intervals each containing at least one prime which is in those arithmetic progressions and for which we have $\lambda_0+\lambda_1F_p(\alpha_1)+\ldots+\lambda_kF_p(\alpha_k) \neq 0$. We also derive an explicit $p$-adic lower bound for the previous linear form. 
\end{abstract}


\section{Introduction}

In the eighteenth century, Leonhard Euler studied factorial series (see. e.g. \cite{V2007}). His interest of these series has also lead us to study \textit{Euler's factorial series} which is defined as follows:
\begin{equation*}
    F(t) \coloneqq {}_2F_0(1,1|t)=\sum_{n=0}^\infty n!t^n.
\end{equation*}
Further, we write $F_p(t)=\sum_{n=0}^\infty n!t^n$ when we consider the Euler's factorial series in $p$-adic domain. It is in a class of series which can be written by $\sum_{n\geq 0} a_n n! t^n$, where $a_n$ satisfies certain properties. These series are called \textit{F-series} and they were introduced by V. G. Chirski\u{i} \cite{C1989,C1990}. Euler's factorial series has also connections to quantum physics (see e.g. \cite{BW2015,F1997}).

It is clear that using standard Archimedean metric, Euler's factorial series converges only when $t=0$. However, using $p$-adic metric, the series converges for all integers $t$. Hence, an interesting question is, how close to zero the term $F_p(t)$ can be using $p$-adic metric. Or even more, if we look at a linear form
\begin{equation}
\label{eq:DefLambda}
    \Lambda_p=\lambda_0+\lambda_1F_p(\alpha_1)+\lambda_2F_p(\alpha_2)+\ldots+\lambda_kF_p(\alpha_k),
\end{equation}
where $k \geq 1$, numbers $\lambda_j$ are integers, $\lambda_j \ne 0$ for at least one index $j$ and numbers $\alpha_j$ are distinct non-zero integers, can $\Lambda_p$ be zero? If not, how close is it to zero in $p$-adic metric?

D. Bertrand, V. G. Chirski\u{i} and J. Yebbou considered these problems for $F$-series over $\mathbb{K}(x)$ where $\mathbb{K}$ is an algebraic number field over $\mathbb{Q}$. In 2004, they showed (see \cite{BCY2004}) that assuming some extra conditions relating to a certain differential system, there is an infinite collection of intervals containing a prime $p$ such that for some valuation $v \mid p$ certain linear forms of the $F$-series are not zero. $v$-adic lower bounds for the linear forms were also given and the results were effective.

Now the question is: can we prove something similar but explicit considering linear forms in Euler's factorial series? May some other conditions help us to maintain these results or be useful to obtain the results? Or do we know any more precise results about the terms $F_p(t)$ which do not equal to zero?

In 2015, V. G. Chirski\u{i} \cite{C2015} proved that there exist infinitely many primes such that $F_p(1) \neq 0$. Three years later, T. Matala-aho and W. Zudilin \cite{MZ2018} showed the following result:
\begin{theorem}
\label{thm:Zudilin}
\cite[Theorem 1]{MZ2018}
Given $t \in \mathbb{Z}\setminus \{0\}$, let $R \subseteq \mathbb{P}$, where $\mathbb{P}$ denotes the set of primes, be such that
\begin{equation}
\label{eq:OldCOndC}
    \limsup_{n \to \infty} c^n n! \prod_{p \in R} \left|n!\right|_p^2 =0, \quad \text{where } c=c(t;R):=4 |t|\prod_{p \in R} |t|_p^2.
\end{equation}
Then either there exists a prime $p \in R$ for which $F_p(t)$ is irrational, or there are two distinct primes $p, q \in R$ such that $F_p(t) \neq F_q(t)$.
\end{theorem}

In 2019, A.-M. Ernvall-Hyt\"onen, T. Matala-aho and L. Sepp\"al\"a \cite{EHMS2019} studied more about these type of conditions when they considered the case where $k =1$ and $\lambda_1 \ne 0$ in formula \eqref{eq:DefLambda}. They showed that if condition \eqref{eq:OldCOndC} is satisfied for all sets $R \setminus S$, where $S$ is any finite subset of $R$, then there exists infinitely many primes $p \in R$ such that $\Lambda_p \neq 0$. This raised a question: what kind of sets could satisfy condition \eqref{eq:OldCOndC}. Hence, in the same article, they showed that a set $R$, which consists of primes in more than $\varphi(m)/2$ residue classes in the reduced residue system modulo $m \geq 3$, satisfies the condition \eqref{eq:OldCOndC} and there are infinitely primes in $R$ such that $\Lambda_p \neq 0$. Even more, they also proved the following, conditional result:
\begin{theorem}
\label{thm:EHMaS}
\cite[Theorem 8]{EHMS2019}
Assume the Generalized Riemann Hypothesis. Let $m \geq 3$ be a given integer. Assume that $R$ is any union of primes in $\varphi(m)/2$ residue classes in reduced residue system modulo $m$. Then there is a value $d_m$ such that if $t$ is any non-zero integer satisfying the bound
\begin{equation*}
    4\left|t\right|\prod_{p \in R} \left|t \right|_p^2<d_m,
\end{equation*}
then there exists a prime number $p \in R$ for which $\lambda_0-\lambda_1F_p(t) \neq 0$. Here $\lambda_0$ and $\lambda_1 \neq 0$ are fixed integers. 
\end{theorem}

Further, very recently A.-M. Ernvall-Hyt\"onen, T. Matala-aho and L. Sepp\"al\"a \cite{EHMS2022} showed an explicit, $p$-adic lower bound for $\Lambda_p$ with $k=1$, $\lambda_1 \ne 0$ and $\alpha_1=\pm p^a$ when $p^a$ large enough. 

In 2020, L. Sepp\"al\"a \cite{S2020} generalized the results for the Euler's factorial series at algebraic integer points and for linear forms $\Lambda_p$ with $k \geq 1$. She showed that if a type \eqref{eq:OldCOndC} condition holds, then for some valuation $v$ we must have $\Lambda_v \neq 0$. She also proved explicit lower bounds for $\left|\Lambda_v\right|_v$ where numbers $p$ are in certain intervals. In addition to proving the previous results for sets of primes, L. Sepp\"al\"a considered primes in arithmetic progressions. She proved that if we consider a set $R$ which consists of primes in more than $k\varphi(m)/(k+1)$ residue classes in reduced residue system modulo $m$, then for some non-Archimedean valuation $v$ satisfying $v \mid p$ for some $p \in R$, we have $\Lambda_v \ne 0$. 

In this article, we first prove that if a type \eqref{eq:OldCOndC} condition holds for a set $R$ and for $\Lambda_p$ with $k \geq 1$, then $\Lambda_p \neq 0$ for some $p \in R$. This is done in Theorem \ref{thm:limSupCond} and the result is a slight improvement to Corollary 9.1 in \cite{S2020}. Then, in Theorem \ref{thm:dmGene} we prove an explicit version of Theorem \ref{thm:EHMaS} and generalize it to cases $k>1$, too. After that, in Theorem \ref{thm:lowerBound} we show explicit $p$-adic lower bounds for the terms $\left|\Lambda_p\right|_p$, where numbers $p$ are in arithmetic progressions in $k\varphi(m)/(k+1)$ residue classes up to some height which depends on the linear form. Finally, in Theorem \ref{thm:lowerBoundMore} we prove $p$-adic lower bounds for $\left|\Lambda_p\right|_p$ also in the case where we take primes in arithmetic progressions in certain intervals from more than $k\varphi(m)/(k+1)$ residue classes. In addition to the main results, some corollaries for functions 
\begin{equation}
\label{def:pitheta}
\pi(x;m,a):=\sum_{\substack{p\leq x \\ p\equiv a\pmod m }} 1, \qquad \theta(x;m,a):=\sum_{\substack{p\leq x \\ p\equiv a\pmod m }}\log{p}
\end{equation}
and related functions are derived in Sections \ref{sec:piTheta} and \ref{sec:EstDepend}. 

Throughout the article $p$ denotes a prime number and $\mathbb{P}$ is a set of prime numbers. We also use abbreviation RH for the Riemann hypothesis and GRH for the Generalized Riemann Hypothesis.

\section{Results}
\label{sec:results}

In this section, we describe some definitions, notations and the main results. 

First we define terms $c_1$ and $c_2$ for integers $k, \alpha_1,\ldots, \alpha_k$ and set $R$:
\begin{equation}
    \label{def:c1}
    c_1\coloneqq c_1(\overline{\alpha})=2^{k}\left(\max_{1\leq j \leq k}\{|\alpha_j|\}\right)^{k} \quad\text{and}\quad c_2\coloneqq c_2(\overline{\alpha}; R)\coloneqq c_1\prod_{p \in R}\left(\max_{1 \leq i \leq k}\{\left|\alpha_i\right|_p\}\right)^{k+1}.
\end{equation}

The first result tells that if a certain condition holds, then a linear form of functions $F_p$ must be non-zero.

\begin{theorem}
\label{thm:limSupCond}
Let $k \geq 1$ be an integer and let $\lambda_0, \lambda_1, \ldots, \lambda_k$ be integers such that $\lambda_i \neq 0$ for at least one index $i$. Further, assume that numbers $\alpha_1,\alpha_2,\ldots, \alpha_k$ are non-zero integers. Let $R$ be a non-empty set of prime numbers such that condition
\begin{equation}
\label{eq:limSupCond}
    \limsup_{n \to \infty} c_2(\overline{\alpha}; R)^n(kn+k+1)!\prod_{p \in R}\left|(kn)!n!\right|_p=0
\end{equation}
holds. Then for at least one prime number $p \in R$ we have
$
    \lambda_0+\lambda_1F_p(\alpha_1)+\ldots+\lambda_kF_p(\alpha_k) \neq 0.
$
\end{theorem}

The previous result raises a question for what kind of non-empty set $R$ the condition \eqref{eq:limSupCond} holds. Next we prove that primes in enough many arithmetic progressions satisfy the condition if one extra assumption for the number $c_2$ is assumed.

\begin{theorem}
\label{thm:dmGene}
Assume that $k \geq 1$, $m \geq 3$ and $\lambda_0, \lambda_1, \ldots,\lambda_k$ be integers such that $\lambda_j \neq 0$ for at least one index $j$. Even more, assume RH and that the $m$th cyclotomic Dedekind zeta function satisfies the GRH. Further, suppose that $R$
is any union of primes in at least $k\varphi(m)/(k+1)$ residue classes in the reduced residue system modulo $m$.

If $\alpha_1, \alpha_2,\ldots, \alpha_k$ are any non-zero integers for which $c_2(\overline{\alpha}; R)$ satisfies the bound
\begin{equation}
\label{eq:cUpper}
   c_2(\overline{\alpha}; R) <\left(ke^{1+6.550\varphi(m)}m^{\varphi(m)}\right)^{-k},
\end{equation}
then there exists a prime number $p \in R$ for which 
\begin{equation*}
    \lambda_0+\lambda_1F_p(\alpha_1)+\ldots+\lambda_kF_p(\alpha_k)\neq 0.
\end{equation*}
\end{theorem}

\begin{remark}
\label{remark:ExistsC}
We notice that there actually are sets $R$ such that condition \eqref{eq:cUpper} holds: For example, let $k \geq 1$ and $m\geq 3$. Further, let $R$ be a set containing all primes which are congruent to $a$ modulo $m$, where $\gcd(a,m)=1$, for at least $k\varphi(m)/(k+1)$ different integers $a$ that are not congruent to each other modulo $m$. Let us set $\alpha_i=ip$ for all indices $i$, where  $p\in R$ such that $p> \left(2k^2e^{1+6.550\varphi(m)}m^{\varphi(m)}\right)^{k}$. We can always find such a prime since there are infinitely many primes that are congruent to $a$ modulo $m$. Then inequality \eqref{eq:cUpper} is satisfied.
\end{remark}

\begin{remark}
It is essential to use number $c_2$ given in \eqref{def:c1} instead of number $c_2$ in \cite[Theorem 3.1]{S2020} since the later one is always at least one and the upper bound obtained in Theorem \ref{thm:dmGene} is smaller than one. This also justifies the need to prove Theorem \ref{thm:limSupCond}.
\end{remark}

Since the previous theorem tells us that under certain conditions there is a prime such that a linear form does not equal to zero, a natural question is, how close it is to zero and how large the possible prime number is. To answer this question, we first denote
\begin{equation}
\label{eq:DefD}
    D(m,k,\overline{\alpha}; R):=-\left(\log c_2(\overline{\alpha}; R)+k\left(\log k+\left(\log m+6.550\right)\varphi(m)+1\right)\right).
\end{equation}
Now we are ready to describe the next result:

\begin{theorem}
\label{thm:lowerBound}
Assume that numbers $k,\lambda_j$ and $m$ and set $R$ are defined as in Theorem \ref{thm:dmGene}. Assume also RH and that the $m$th cyclotomic Dedekind zeta function satisfies the GRH and that $\alpha_1,\alpha_2, \ldots, \alpha_k$ are pairwise distinct non-zero integers such that $c(\overline{\alpha}; R')$, where
\begin{equation*}
    R':=R\cap\left[2, k\left(\max\{1865,m^{\varphi(m)}, \max_{1 \leq j \leq k} \{|\alpha_j|\}+2\right) \right],
\end{equation*} 
is smaller than the right-hand side of inequality \eqref{eq:cUpper}.
Suppose also that
$\max_{0 \leq i\leq k} \{\left|\lambda_i\right|\} \leq H$, where 
\begin{equation}
\label{eq:BoundH}
    \frac{\log H}{D(m,k,\overline{\alpha}; R')} 
    \geq \max\left\{1866, m^{\varphi(m)}+1,\max_{1\leq j \leq k}\{|\alpha_j|\}+1, e^{\frac{2\left(2.539k\varphi(m)+5.724k+0.054\right)}{D(m,k,\overline{\alpha}; R')}}+1\right\}.
\end{equation} 
Further, let 
\begin{equation}
\label{eq:interval}
R'':=R \cap \left[2,\frac{2k\log H}{D(m,k,\overline{\alpha}; R')}+2k\right).
\end{equation}

Then, for some $p \in R''$ we have
\begin{multline*}
    \left|\lambda_0+\lambda_1F_p(\alpha_1)+\ldots+\lambda_kF_p(\alpha_k)\right|_p\\
    >H^{-\left(\frac{2k\log\log H}{D(m,k,\overline{\alpha}; R')}+\frac{-2k\log D(m,k,\overline{\alpha}; R')+2k^2+3k\log 2-2k}{D(m,k,\overline{\alpha}; R')}+\frac{4k\left(2.539k\varphi(m)+5.724k+0.054\right)}{D(m,k,\overline{\alpha}; R')^2}+0.015k^2+ 0.235 k+1.034 \right)}.
\end{multline*}
\end{theorem}

\begin{remark}
According to Remark \ref{remark:ExistsC}, taking suitable numbers $\alpha_j$ and a large enough $H$ it is possible to find number $c_2(\overline{\alpha}; R')$ such that inequality \eqref{eq:cUpper} holds and bound \eqref{eq:BoundH} is satisfied.
\end{remark}

The previous lower bound is quite long and includes the term $D(m,k,\overline{\alpha}; R')$ which may be difficult to determine. Hence, we prove a little bit weaker but simpler result:

\begin{corollary}
\label{corollary:lowerBound}
Let us use the same definitions and assumptions as in Theorem \ref{thm:lowerBound}. Then, for some prime number $p \in R'$ we have 
\begin{equation*}
    \left|\lambda_0+\lambda_1F_p(\alpha_1)+\lambda_2F_p(\alpha_2)+\ldots+\lambda_kF_p(\alpha_k)\right|_p>H^{-\left(0.0008 k^2+0.055k+0.675\right)(\log\log H)^2}.
\end{equation*}
\end{corollary}

In the last two results the exponent of $H$ goes to minus infinity when H goes to infinity. As a final main result, we prove that if we consider more than $k\varphi(m)/(k+1)$ residue classes, then
under certain conditions we obtain a different lower bound where the exponent of $H$ stays
bounded as $H$ diverges. The result is divided into two cases depending on $k$ since there is clear improvement in the result when $k \geq 2$ and some partial results in the case $k=1$ are already known (see \cite[Corollaries 1.5, 1.6 and 1.7]{EHMS2022}).

\begin{theorem}
\label{thm:lowerBoundMore}
Assume that $m, k$ and $\lambda_0,\ldots, \lambda_k$ satisfy the same hypothesis as in Theorem \ref{thm:dmGene} and let $\alpha_1,\ldots, \alpha_k$ be distinct non-zero integers. Assume also RH  and that the $m$th cyclotomic Dedekind zeta function satisfies the GRH, that $\varepsilon \in (0,1)$ is a real number and that the set $R$ consists of primes in arithmetic progressions in at least $(k+\varepsilon)\varphi(m)/(k+1)$ residue classes modulo $m$ and that we have $\max_{0 \leq i\leq k} \{\left|\lambda_i\right|\} \leq H$ where 
\begin{equation}
\label{eq:HassumptionEpsilon}
    \frac{\log H}{\varepsilon} \geq se^s
\end{equation}
and
\begin{equation}
\label{eq:LowerBoundsForS}
    s \geq \max \left\{1866,m^{\varphi(m)}+1, c_1+1, \left(\frac{15.195 k+16.195}{\varepsilon}\right)^{1.5}+1 \right\}.
\end{equation}

Then there exists a prime 
\begin{equation}
\label{eq:intervalEpsilon}
    p \in R\cap\left(\log\left(\frac{\log H}{\varepsilon\log\left(\frac{\log H}{\varepsilon}\right)}\right),k\left(\frac{12.994\log H}{\varepsilon\log\left(\frac{\log H}{\varepsilon}\right)}+2\right)\right)
\end{equation}
such that 
\begin{multline*}
    \left|\lambda_0+\lambda_1F_p(\alpha_1)+\lambda_2F_p(\alpha_2)+\ldots+\lambda_kF_p(\alpha_k)\right|_p \\
    >\begin{cases}
    H^{-\left(\frac{k}{\varepsilon}+1\right)-\frac{\left(197.444k^2/\varepsilon+210.438k/\varepsilon+1.726k^2-0.529k+12.995\right)\log\log(0.072\log H)}{\varepsilon \log\log H}} &\text{if } k \geq 1 \\
    H^{-\left(\frac{k}{\varepsilon}+1\right)-\frac{\left(97.932k^2/\varepsilon+104.377k/\varepsilon+0.856k^2-0.262k+6.446\right)\log\log(0.02\log H)}{\varepsilon \log\log H}} &\text{if } k \geq 2
    \end{cases}.
\end{multline*}
\end{theorem}

\begin{remark}
Due to assumptions \eqref{eq:HassumptionEpsilon} and \eqref{eq:LowerBoundsForS}, we have
\begin{equation*}
    \log H >\varepsilon\left(\frac{31^{1.5}}{\varepsilon^{1.5}}+1\right)e^{\frac{31^{1.5}}{\varepsilon^{1.5}}+1}\geq 31^{1.5}e^{31^{1.5}+1}+e^{31^{1.5}+1}>10^{70}.
\end{equation*}
Hence, terms $\log\log(0.072\log H)$ and $\log\log(0.02\log H)$ in Theorem \ref{thm:lowerBoundMore} are positive.
\end{remark}

\section{Outline of the proofs}
\label{sec:outline}
In this section, we describe the main idea of the proof. As in \cite{EHMS2019,M2016,S2020} we start by estimating the functions $F(\alpha_j)$ using Pad\'e approximations. Thus we write
\begin{equation*}
   B_{n,\mu,0}(t)F(\alpha_jt)-B_{n,\mu,j}(t)=S_{n,\mu,j}(t), 
\end{equation*}
where $n \in \mathbb{Z_+}$, $\mu \geq 0$, $\mu \in \mathbb{Z}$, for each $j=1,2\ldots, k$ for certain explicit functions $B_{n, \mu, 0}$, $B_{n,\mu, j}$ and $S_{n, \mu, j}$. Then we have
\begin{equation}
\label{eq:defT}
    B_{n,\mu,0}(1)\Lambda_p=\sum_{j=0}^k \lambda_jB_{n,\mu,j}(1)+\sum_{j=1}^k \lambda_jS_{n,\mu,j}(1) \coloneqq T(n, \mu)+\sum_{j=1}^k \lambda_jS_{n,\mu,j}(1).
\end{equation}

Now, let $R$ be a set of certain primes in arithmetic progressions. It is clearly non-empty in Theorems \ref{thm:limSupCond} and \ref{thm:dmGene} and it is proved to be non-empty in Theorems \ref{thm:lowerBound} and \ref{thm:lowerBoundMore} (see Section \ref{sec:Existence}). If we have $\Lambda_p=0$ for all $p \in R$ (i.e. contradiction to Theorem \ref{thm:limSupCond}), from which it follows that
\begin{equation*}
    \left|T(n+1,\mu)\right|_p=\left|\sum_{j=1}^k \lambda_jS_{n+1,\mu,j}(1)\right|_p,
\end{equation*}
or if we have
\begin{equation}
\label{eq:LambdaContra}
    \left|B_{n+1,\mu,0}(1)\Lambda_p\right|_p \leq \left|\sum_{j=1}^k \lambda_jS_{n+1,\mu,j}(1)\right|_p 
\end{equation}
for all $p \in R$, then estimate
\begin{align}
    1&=\left|T(n+1, \mu)\right|\prod_{p \in \mathbb{P}} \left|T(n+1, \mu)\right|_p \leq \left|T(n+1, \mu)\right|\prod_{p \in R} \left|T(n+1, \mu)\right|_p \nonumber\\ 
    &\quad\leq (k+1)\max_{0 \leq i\leq k} \{\left|\lambda_i\right|\} \max_{0 \leq i \leq k} \{\left|B_{n+1,\mu,i}(1)\right|\}\cdot \max_{1 \leq i \leq k} \left\{ \prod_{p \in R}\left|S_{n+1,\mu,i}(1)\right|_p\right\} \label{eq:Product}
\end{align}
holds. In order to prove Theorems \ref{thm:limSupCond}, \ref{thm:lowerBound} and \ref{thm:lowerBoundMore} and Corollary \ref{corollary:lowerBound}, we show that the right-hand side of inequality \eqref{eq:Product} is smaller than $1$ for some integer $n$ when $\mu$ is selected in such a way that 
\begin{equation}
\label{eq:defMu}
    T(n+1,\mu) \neq 0, \quad \mu \in \{0,1,\ldots, k\}
\end{equation}
(see \cite[Lemma 6.2]{S2020}). This is a contradiction. Hence neither $\Lambda_p =0$ nor inequality \eqref{eq:LambdaContra} cannot hold for all $p \in R$. The first one proves Theorem \ref{thm:limSupCond} and Theorem \ref{thm:dmGene} follows from that.

In order to prove Theorems \ref{thm:lowerBound} and \ref{thm:lowerBoundMore} and Corollary \ref{corollary:lowerBound} we have to do a little bit more work. Since there must exists $p \in R$ for which inequality \eqref{eq:LambdaContra} cannot hold, we must have
\begin{equation}
\label{eq:TnEnough}
    1 \leq \left|T(n+1, \mu)\right|\, \left|T(n+1, \mu)\right|_p=\left|T(n+1, \mu)\right|\, \left|B_{n+1,\mu,0}(1)\Lambda_p\right|_p \leq \left|T(n+1, \mu)\right|\,\left|\Lambda_p\right|_p. 
\end{equation}
Hence, in order to find lower bounds for the terms $\left|\Lambda_p\right|_p$, it is sufficient to find and upper bound for the term $\left|T(n, \mu)\right|$.

The estimates for the Pad\'e approximations are derived in Section \ref{sec:Pade}. In Sections \ref{sec:piTheta} and \ref{sec:EstDepend} we prove some results related to primes in arithmetic progressions which are used later in our main proofs. In Section \ref{sec:Existence}, we show that there exists primes in sets \eqref{eq:interval} and \eqref{eq:intervalEpsilon} in Theorems \ref{thm:lowerBound} and \ref{thm:lowerBoundMore}. In Sections \ref{sec:proof1} and \ref{sec:proof2}, we prove Theorems \ref{thm:limSupCond} and \ref{thm:dmGene} respectively. In Sections \ref{sec:Contra} and \ref{sec:Contra2} we show the contradictions that the right-hand side of inequality \eqref{eq:Product} must be smaller than one, first in a setup of Theorem \ref{thm:lowerBound} and then with respect to Theorem \ref{thm:lowerBoundMore}. Finally, in Sections \ref{sec:proof3} and \ref{sec:proof4}, we prove Theorems \ref{thm:lowerBound} and \ref{thm:lowerBoundMore} and Corollary \ref{corollary:lowerBound}.

\section{Preliminary estimates}
\subsection{Estimates for Pad\'e approximations}
\label{sec:Pade}

In this section, we consider Pad\'e approximations for the Euler factorial series and prove estimates for them. 

The following function is used in the definitions:
\begin{equation}
\label{eq:defSigma}
    \sigma_i \coloneqq \sigma_i(n, \overline{\alpha})=(-1)^i\sum_{\substack{i_1+\ldots+i_k=i \\ 0 \leq i_j \leq n}} \binom{n}{i_i}\binom{n}{i_2}\cdots \binom{n}{i_k}\alpha_1^{n-i_i}\alpha_2^{n-i_2}\cdots \alpha_k^{n-i_k}.
\end{equation}

Now we describe the Pad\'e approximations we are going to use:
\begin{theorem}
\cite[Theorem 4.2 and Section 4.2]{S2020}
\label{thm:Pade}
Let $k, \alpha_i$ be defined as Theorem \ref{thm:lowerBound}, and assume that $\mu \geq 0$, $\mu \in \mathbb{Z}$ and $n \in \mathbb{Z}_+$. Then we have
\begin{equation*}
    B_{n,\mu,0}(t)F(\alpha_jt)-B_{n,\mu,j}(t)=S_{n,\mu,j}(t), \quad j=1,2,\ldots, k,
\end{equation*}
where
\begin{align*}
        & B_{n,\mu,0}(t):=\sum_{i=0}^{kn} \sigma_i\cdot\frac{(kn+\mu)!}{(i+\mu)!}t^{kn-i} \\
        & B_{n,\mu,j}(t):=(kn+\mu)!\sum_{N=0}^{kn+\mu-1}t^N\sum_{h=0}^{\min\{kn,N\}}\sigma_{kn-h}\cdot\frac{(N-h)!}{(kn+\mu-h)!}\cdot\alpha_j^{N-h} \\
        & S_{n,\mu,j}(t):=(kn+\mu)!n!t^{(k+1)n+\mu}\sum_{h=0}^\infty h!\binom{n+h}{h}\alpha_j^{n+h+\mu}t^h\sum_{i=0}^{kn}\sigma_i\binom{i+\mu+n+h}{i+\mu}\alpha_j^i.
\end{align*}
\end{theorem}

Next we estimate the approximations. In order to do it, we need the following lemma:

\begin{lemma}
\label{lemma:sigma}
\cite[Lemma 4.1, estimate (15)]{S2020}
Let $k$ be defined as Theorem \ref{thm:lowerBound}, $n\in \mathbb{Z}_+$ and $\sigma_i(n, \overline{\alpha})$ be defined as in \eqref{eq:defSigma}. Then for all $t \geq 0$ we have
\begin{equation*}
    \sum_{i=0}^{kn}\left|\sigma_i(n, \overline{\alpha})\right|t^i \leq \prod_{i=1}^k \left(\left|\alpha_i\right|+t\right)^n.
\end{equation*}
\end{lemma}

Now we are ready to estimate the Pad\'e approximations:
\begin{lemma}
\label{lemma:BSEstimates}
Let $k, \alpha_i$ be defined as Theorem \ref{thm:lowerBound} and let terms $\mu, n$, $B_{n,\mu,0}(1), B_{n,\mu,j}(1)$ and $S_{n,\mu,j}(1)$ be as in Theorem \ref{thm:Pade}. Then we have
\begin{align*}
        & \left|B_{n,\mu,0}(1)\right|\leq (kn)!\binom{kn+\mu}{\mu}\prod_{i=1}^{k}\left(\left|\alpha_i\right|+1\right)^n,  \\
        & \left|B_{n,\mu,j}(1)\right| \leq (kn+\mu)!\cdot(kn+\mu)\cdot |\alpha_j|^{\mu-1} \prod_{i=1}^k \left(\left|\alpha_i\right|+\left|\alpha_j\right|\right)^n \\
        &\text{and} \quad \left|S_{n,\mu,j}(1)\right|_p \leq \left|(kn+\mu)!n!\right|_p\cdot \left(\max\{\left|\alpha_j\right|_p\}\right)^{(k+1)n}.
\end{align*}
\end{lemma}

\begin{proof}
The estimate for the term $B_{n,\mu,0}(1)$ follows from \cite[Section 7]{S2020}. So let us move on to estimate the term $B_{n,\mu,j}(1)$. Adding non-negative terms if needed and using Lemma \ref{lemma:sigma} we can deduce
\begin{align*}
    \left|B_{n,\mu,j}(1)\right|&=\left|(kn+\mu)!\sum_{N=0}^{kn+\mu-1}\sum_{h=0}^{\min\{kn,N\}}\sigma_{kn-h}\cdot\frac{(N-h)!}{(kn+\mu-h)!}\cdot\alpha_j^{N-h}\right| \\
    &\leq (kn+\mu)!\sum_{N=0}^{kn+\mu-1}\left|\alpha_j\right|^{N-kn}\sum_{h=0}^{\min\{kn,N\}} \left|\sigma_{kn-h}\right|\cdot \left|\alpha_j\right|^{kn-h} \\
    &\leq (kn+\mu)!\cdot(kn+\mu)\cdot |\alpha_j|^{\mu-1} \prod_{i=1}^k \left(\left|\alpha_i\right|+\left|\alpha_j\right|\right)^n,
\end{align*}
which proves the second case.

Now we have only the term $S_{n,\mu,j}(1)$ left. Keeping in mind that the numbers $\alpha_1,\alpha_2,\ldots, \alpha_k$ are integers and 
using the estimate
\begin{equation*}
   \left|\sigma_i\right|_p = \left|\sum_{i_1+\ldots+i_k=i} \binom{n}{i_i}\binom{n}{i_2}\cdots \binom{n}{i_k}\alpha_1^{n-i_i}\alpha_2^{n-i_2}\cdots \alpha_k^{n-i_k}\right|_p \leq \left(\max\{\left|\alpha_j\right|_p\}\right)^{kn-i},
\end{equation*}
we get
\begin{align*}
    \left|S_{n,\mu,j}(1)\right|_p&=\left|(kn+\mu)!n!\sum_{h=0}^\infty h!\binom{n+h}{h}\alpha_j^{n+h+\mu}\sum_{i=0}^{kn}\sigma_i\binom{i+\mu+n+h}{i+\mu}\alpha_j^i\right|_p \\
    &\leq \left|(kn+\mu)!n!\right|_p\cdot \max_{\substack{0\leq h <\infty, \\ 0\leq i \leq kn}}\left| h!\binom{n+h}{h}\binom{i+\mu+n+h}{i+\mu}\right|_p \left(\max\{\left|\alpha_j\right|_p\}\right)^{(k+1)n+\mu} \\
    & \leq \left|(kn+\mu)!n!\right|_p\cdot \left(\max\{\left|\alpha_j\right|_p\}\right)^{(k+1)n}.
\end{align*}
Hence we have proved the estimate for the term $S_{n,\mu,j}(1)$ too.
\end{proof}

\subsection{Estimates for the functions $\pi(x;m,a)$ and $\theta(x;m,a)$}
\label{sec:piTheta}
In this section we derive estimates for primes in arithmetic progressions and for related functions. Indeed, we consider functions $\pi(x;m,a)$ and $\theta(x;m,a)$ which are given in \eqref{def:pitheta}.
The estimates are later used in Section \ref{sec:EstDepend} to derive estimates for some sums which are used to prove the main results.

Recall that $\mathrm{Li}(x):=\int_2^x (\log{t})^{-1} \, dt$. First, we mention three results which are used to prove our estimates for functions $\pi(x;m,a)$ and $\theta(x;m,a)$:
\begin{theorem}
\label{pi}
\cite[Special case of Corollary 1]{GM2019} Let $a$ and $m \geq 1$ be integers and $x \geq 2$ be a real number such that $\gcd(a, m) = 1$. Assume that the $m$th cyclotomic Dedekind zeta function satisfies the GRH. Then we have
\begin{equation*}
    \left|\pi(x;m,a)-\frac{\mathrm{Li}(x)}{\varphi(m)}\right| \leq \sqrt{x}\left(\frac{\log x}{8\pi}+\left(\frac{1}{2\pi}+\frac{3}{\log x}\right)\log m+\frac{1}{4 \pi}+\frac{6}{\log x}\right).
\end{equation*}
\end{theorem}

\begin{lemma}
\label{lemma:Var}
\cite[Theorem 15, estimates (3.41) and (3.42)]{RS1962}
Let $m$ be an integer and $\gamma$ be Euler's constant. For $3 \leq m$ we have
\begin{equation*}
    \frac{m}{\varphi(m)}<e^\gamma \log\log m+\frac{2.50637}{\log\log m}.
\end{equation*}
\end{lemma}

\begin{corollary}
\label{corollary:varLower}
Let $m\ge 3$, $m \not\in\{3,4,6,8,10,12,14,18,20,24,30,36,42,60\}$
be an integer. Then $\varphi(m)>m^{0.7}$.
\end{corollary}

\begin{proof}
Let us first assume that $m \geq 212$. Then $2.50637/\log\log m<0.5e^\gamma\log\log m$, $1.5e^\gamma\log\log m<m^{0.3}$ and by Lemma \ref{lemma:Var} the claim holds for $m \ge 212$. Further, we check numerically the cases $m \leq 211$ where we find that the claim holds with only those exceptions.
\end{proof}

Now we derive an estimate for $\pi(x;m,a)$ when $x$ is large enough with respect to $m$:
\begin{lemma}
\label{lemma:estPi}
Suppose the same assumptions as in Theorem \ref{pi}. Assume also that $m \geq 3$ and $x \geq \max\{m^{\varphi(m)},1865\}$. Then
\begin{equation*}
\left|\pi(x;m,a)-\frac{\mathrm{Li}(x)}{\varphi(m)}\right| \leq \frac{\sqrt{x}\log x}{8\pi}+ \sqrt{x}\left(\frac{\log m}{2\pi}+\frac{3}{\varphi(m)}+0.341\right).
\end{equation*}
\end{lemma}

\begin{proof}
Let us first consider case $3 \leq m \leq 10^5$ and $x \leq 10^{11}$. Since $x \geq 1865$, by \cite[Theorem 1.9]{BGOR2018} we have 
\begin{equation}
\label{eq:piSmallValuesEst}
    \left|\pi(x;m,a)-\frac{\mathrm{Li}(x)}{\varphi(m)}\right| \leq 
    2.734\frac{\sqrt{x}}{\log x} 
    <\frac{\sqrt{x}\log x}{8\pi}+0.064\sqrt{x}.
\end{equation}
Let us now consider the rest of the cases. 

The rest of the cases follow easily from Theorem \ref{pi}: Using the theorem and the assumption $x \geq m^{\varphi(m)}$, we can deduce that
\begin{align}
    &\left|\pi(x;m,a)-\frac{\mathrm{Li}(x)}{\varphi(m)}\right|\leq \frac{\sqrt{x}\log x}{8\pi}+ \sqrt{x}\left(\frac{\log m}{2\pi}+\frac{3}{\varphi(m)}+\frac{1}{4 \pi}+\frac{6}{\log x}\right)\label{eq:pivar}.
\end{align}
Suppose that $x \geq 10^{11}$, or that $x \geq m^{\varphi(m)}$ with $m \geq 10^5$. Then $x \geq 10^{11}$ in any case (in the second case this follows from Corollary \ref{corollary:varLower}) and the assumptions that $m \geq 10^5$ and $x \geq m^{\varphi(m)}$). Under this hypothesis $1/(4\pi)+6/\log x \leq 0.341$ and the claim follows from \eqref{eq:piSmallValuesEst} and \eqref{eq:pivar}.
\end{proof}

Next we prove estimates for function $\theta(x;m,a)$. In order to prove those estimates, we use function
\begin{equation*}
\psi(x;m,a) := \sum_{\substack{p^k \leq x \\ p^k \equiv a \pmod m}} \log p.
\end{equation*}
First, we mention a useful estimate for the function $\psi(x;m,a)$:

\begin{theorem}
\label{psi} 
\cite[Special case of Theorem 1]{GM2019} Let $a$ and $m \geq 1$ be integers such that $\gcd(a, m) = 1$ and let $x \geq 1$ be a real number. Assume that the $m$th cyclotomic Dedekind zeta function satisfies the GRH. Then we have
\begin{equation*}
   \left|\psi(x;m,a)-\frac{x}{\varphi(m)}\right| \leq \sqrt{x}\left(\frac{\log^2 x}{8\pi}+\left(\frac{\log x}{2\pi}+2\right)\log m+2\right).
\end{equation*}
\end{theorem}

Now we denote
\begin{equation*}
    a(m):=
    \begin{cases}
     0.129  &\text{ if } 3\leq m\leq 432 \\
     0.399 &\text{ if } m \geq 433 
    \end{cases}.
\end{equation*}
and move on to estimate function $\theta(x;m,a)$. 

\begin{lemma}
\label{corollary:theta}
Let $x,m$ and $a$ be as in Theorem \ref{psi}, $m \geq 3$ and $x \geq \max\{\sqrt{1865},m\}$. Assume that the $m$th cyclotomic Dedekind zeta function satisfies the GRH. Then we have 
\begin{equation*}
   \left|\theta(x;m,a)-\frac{x}{\varphi(m)}\right| \leq  a(m)\sqrt{x}\log^2 x.
\end{equation*}
\end{lemma}

\begin{proof}
First we notice that as in the third paragraph of the proof of Theorem 1 in \cite{EHP2022}, we have
\begin{equation}
\label{eq:psiEctraTerm}
   0\leq  \psi(x;m,a)-\theta(x;m,a)<1.4262\sqrt{x}.
\end{equation}
Hence we could use the bound given in Theorem \ref{psi} alone. However, applying also other known results we are able to derive a sharper result for the claim.

Let us first consider the case $3 \leq m \leq 72$. First, if $\sqrt{1865}\leq x \leq 10^{10}$, then by \cite[Table 2]{RR1996} we have
\begin{equation*}
    \left|\theta(x;m,a)-\frac{x}{\varphi(m)}\right|<1.817557\sqrt{x}\leq \frac{1.817557}{(\log \sqrt{1865})^2}\sqrt{x}\log^2 x<0.129\sqrt{x}\log^2 x.
\end{equation*}
Similarly, if $x \geq 10^{10}$, then by \cite[Corollary 3.3]{D1998} and estimate \eqref{eq:psiEctraTerm}, we get
\begin{multline*}
    \left|\theta(x;m,a)-\frac{x}{\varphi(m)}\right|<\frac{11}{32\pi}\sqrt{x}\log^2 x+1.4262\sqrt{x}=\left(\frac{11}{32\pi}+\frac{1.4262}{\left(\log(10^{10})\right)^2}\right)\sqrt{x}\log^2 x \\
    < 0.113\sqrt{x}\log^2 x. 
\end{multline*}
Hence, the claim holds in the case $3 \leq m \leq 72$.

Let us now consider the case $73 \leq m \leq 10^5$. By \cite[Theorem 1.9]{BGOR2018}
\begin{equation}
\label{est:theta73145Third}
    \left|\theta(x;m,a)-\frac{x}{\varphi(m)}\right|<2.072\sqrt{x}\leq \frac{2.072}{(\log 73)^2}\sqrt{x}\log^2 x<0.113\sqrt{x}\log^2 x
\end{equation}
for $x \leq 10^{11}$. Further, if $x \geq 10^{11}$ and $m \leq 432$, then by estimate \eqref{eq:psiEctraTerm} and \cite[Theorem 3.7]{D1998}, we get
\begin{equation*}
    \left|\theta(x;m,a)-\frac{x}{\varphi(m)}\right|<1.4262\sqrt{x}+\frac{11}{32\pi}\sqrt{x}\log^2 x <0.112\sqrt{x}\log^2 x.
\end{equation*}
Hence, in the cases $73 \leq m \leq 432$ the claim holds, too. However, we still have to consider cases $433\leq m$ and $10^{11}<x$ as well as $m >10^5$.

In both cases we have $x\geq m > 10^5$. In addition to that, we apply Theorem \ref{psi} and estimate \eqref{eq:psiEctraTerm} and get
\begin{align*}
    &\left|\theta(x;m,a)-\frac{x}{\varphi(m)}\right|<\sqrt{x}\left(\frac{\log^2 x}{8\pi}+\left(\frac{\log x}{2\pi}+2\right)\log m+3.4262\right) \\
    &\quad< \sqrt{x}\log^2 x\left(\frac{1}{8\pi}+\frac{1}{2\pi}+\frac{2}{\log(10^5)}+\frac{3.4262}{\log^2(10^5)}\right) \\
    &\quad < 0.399\sqrt{x}\log^2 x,
\end{align*}
which proves the claim.
\end{proof}

Now we prove the following estimate for $\theta(x;m,a)$ when $x \geq \max\{m^{\varphi(m)}, 1865\}$:

\begin{lemma}
\label{corollary:theta2}
Let $x,m$ and $a$ be as in Theorem \ref{psi} and let also $m \geq 3$ and $x \geq \max\{m^{\varphi(m)}, 1865\}$. Assume that the $m$th cyclotomic Dedekind zeta function satisfies the GRH. Then we have
\begin{equation*}
    \left|\theta(x;m,a)-\frac{x}{\varphi(m)}\right|<0.092\sqrt{x}\log^2 x.
\end{equation*}
\end{lemma}

\begin{proof}
The claim follows similarly as the results in Lemma \ref{lemma:estPi} and Lemma \ref{corollary:theta}.

Let us first consider the case $3 \leq m \leq 432$. Suppose that $m$ belongs to the even shorter range $3 \leq m \leq 72$ and that $1865 \leq x \leq 10^{10}$.
Then by \cite[Table 2]{RR1996} we have
\begin{equation*}
    \left|\theta(x;m,a)-\frac{x}{\varphi(m)}\right|<1.817557\sqrt{x}\leq \frac{1.817557}{\log^2 1865}\sqrt{x}\log^2 x<0.033\sqrt{x}\log^2 x.
\end{equation*}
For the case $x > 10^{10}$, we use Theorem \ref{psi} and keeping in mind $m \leq 72$ we get
\begin{align}
    &\left|\theta(x;m,a)-\frac{x}{\varphi(m)}\right| \nonumber\\
    & \quad\leq \sqrt{x}\log^2 x\left(\frac{1}{8\pi}+\left(\frac{1}{2\pi \log(10^{10})}+\frac{2}{\log^2(10^{10})}\right)\log 72+\frac{3.4262}{\log^2(10^{10})}\right) \label{eq:replaceXm}\\
   &\quad <0.092\sqrt{x}\log^2 x. \nonumber
\end{align}
Further, replacing lower bound $x \geq 73$ with $x\geq m^{\varphi(m)} \geq  73^{73^{0.7}}$ (see Corollary \ref{corollary:varLower}) in estimates \eqref{est:theta73145Third} and \eqref{eq:replaceXm} and also $72$ with $432$ in estimate \eqref{eq:replaceXm}, this coefficient can be used also in the case $73 \leq m \leq 432$.

Now we move on to the case $m \geq 433$. Using Theorem \ref{psi}, estimate \eqref{eq:psiEctraTerm}, the assumption $x \geq m^{\varphi(m)}$ and Corollary \ref{corollary:varLower} we have
\begin{multline*}
    \left|\theta(x;m,a)-\frac{x}{\varphi(m)}\right|<\sqrt{x}\log^2 x\left(\frac{1}{8\pi}+\frac{1}{2\pi \cdot 433^{0.7}}+\frac{2}{433^{1.4}\log 433}+\frac{3.4262}{433^{1.4}\log^2 433}\right) \\
    <0.043\sqrt{x}\log^2 x.
\end{multline*}
Hence, we can use the bound $0.092$ in every case.
\end{proof}

\subsection{Estimates for some functions depending on the functions $\pi(x;m,a)$ and $\theta(x;m,a)$}
\label{sec:EstDepend}

In this section, we prove estimates for some functions closely related to the functions $\pi(x;m,a)$ and $\theta(x;m,a)$. These estimates are used to prove the main results. 

First we prove results concerning to sum running over primes. 

\begin{lemma}
\label{lemma:estimateSumPNLog}
Let $m, a$ and $x$ be as in Theorem \ref{psi}. Assume also that $m\geq 3$ and $x \geq \max\{m^{\varphi(m)}, 1865\}$, RH holds and that the $m$th cyclotomic Dedekind zeta function satisfies the GRH. Then we have 
\begin{equation*}
      \left|-\sum_{\substack{ p \leq x \\p \equiv a \pmod m}} \frac{\log p}{p-1}+\frac{\log x}{\varphi(m)}\right|<6.550+\log m.
\end{equation*}
\end{lemma}

\begin{proof}
By Abel's summation formula we can write
\begin{align}
    -\sum_{\substack{p \leq x \\p \equiv a \pmod m}}\frac{\log p}{p-1}&=-\frac{\theta(x;m,a)}{x-1}-\int_{2}^x \frac{\theta(y;m,a)}{(y-1)^2} \,dy \nonumber \\
    &= -\frac{\theta(x;m,a)}{x-1}-\int_{2}^{M} \frac{\theta(y;m,a)}{(y-1)^2} \,dy-\int_{M}^x \frac{\theta(y;m,a)}{(y-1)^2} \,dy, \label{eq:BigEst}
\end{align}
where $M=\max\{m, \sqrt{1865}\}$. Note that hence $x \geq M$.
We consider each of the previous three terms on the last three lines of the equality. 

First, by Lemma \ref{corollary:theta2}, in the first term we have
\begin{equation*}
-\left(\frac{x}{\varphi(m)}+0.092\sqrt{x}\log^2 x\right)< -\theta(x;m,a) < -\left(\frac{x}{\varphi(m)}-0.092\sqrt{x}\log^2 x\right).
\end{equation*}
Please notice that since all of the terms on the last two lines of computations \eqref{eq:BigEst} are non-positive, absolute value of the largest error term depends on the case 
$
-0.092\sqrt{x}\log^2 x/(x-1).
$

Similarly, now using Lemma \ref{corollary:theta}, the third term on the right-hand side of equality \eqref{eq:BigEst} is between the values
\begin{equation*}
-\int_{M}^x \frac{y}{(y-1)^2\cdot \varphi(m)} \,dy \pm \int_{M}^x \frac{a(m)\sqrt{y}\log^2 y}{(y-1)^2} \,dy
\end{equation*}
and we are interested in the case where there is a minus sign in front of the second integral.
The first integral is
\begin{equation*}
    \int_{M}^x \frac{y}{(y-1)^2} \,dy =\left[\frac{1}{(1-y)}+\log(y-1)\right]_M^x=\log(x-1)-\log(M-1)+\frac{1}{1-x}-\frac{1}{1-M}
\end{equation*}
and since $M \geq \sqrt{1865}$, the other one is
\begin{align*}
    \left| \int_{M}^x \frac{\sqrt{y}\log^2 y}{(y-1)^2} \,dy\right| &< \frac{2.096}{2}\int_{M}^x \frac{\log^2 y}{y^{1.5}} \,dy \\ &=2.2096\left(\frac{\log^2 M+4\log M+8}{\sqrt{M}}-\frac{\log^2 x+4\log x+8}{\sqrt{x}}\right).
\end{align*}

Let us now move on to the second term in the right-hand side of equation \eqref{eq:BigEst}. Under Riemann Hypothesis, by \cite[Theorem 10, (6.5)]{S1976} we have
\begin{equation}
\label{eq:y}
   0\leq \theta(y;m,a) \leq \sum_{p \leq y} \log p<y+\frac{\sqrt{y}\log^2 y}{8\pi}
\end{equation}
for all $y>0$. Hence, in order to estimate the second term in \eqref{eq:BigEst}, it is sufficient to estimate the integral
\begin{multline*}
    \int_2^M \frac{y+\sqrt{y}\log^2 y/(8\pi)}{(y-1)^2} \, dy\leq \int_2^M \left(\frac{y}{(y-1)^2}+\frac{4\log^2 y}{8\pi y^{1.5}}\right) \, dy \\
    =\left[\frac{1}{1-y}+\log(y-1)-\frac{1}{\pi}\left(\frac{\log^2 y+4\log y+8}{\sqrt{y}}\right)\right]_{2}^M \\
    =-\frac{1}{M-1}+\log(M-1)+1-\frac{1}{\pi}\left(\frac{\log^2 M+4\log M+8}{\sqrt{M}}-\frac{\log^2 2+4\log 2+8}{\sqrt{2}}\right).
\end{multline*}
Thus we have estimated all of the terms in \eqref{eq:BigEst}.

Putting everything together we get 
\begin{align}
   & \left|-\sum_{\substack{ p \leq x \\p \equiv a \pmod m}} \frac{\log p}{p-1}+\frac{\log x}{\varphi(m)}\right|<\left|-\frac{0.092\sqrt{x}\log^2 x}{x-1}+\frac{1}{M-1}-\log(M-1)\right. \nonumber \\
       &\quad \left. -1-\left(\frac{x-1}{x-1}+\log(x-1)-\log x-\log(M-1)-\frac{1}{1-M}\right)\frac{1}{\varphi(m)} \right. \nonumber\\
        &\quad\left. +\frac{1}{\pi}\left(\frac{\log^2 M+4\log M+8}{\sqrt{M}}-\frac{\log^2 2+4\log 2+8}{\sqrt{2}}\right) \right. \nonumber\\
       &\quad \left. -2.096a(m)\left(\frac{\log^2 M+4\log M+8}{\sqrt{M}}-\frac{\log^2 x+4\log x+8}{\sqrt{x}}\right) \right| \nonumber\\
        &\quad<\left|-\frac{0.092\sqrt{x}\log^2 x}{x-1}+\frac{\varphi(m)-1}{(M-1)\varphi(m)}+\log(M-1)\cdot\frac{1-\varphi(m)}{\varphi(m)}-1-\frac{1}{\varphi(m)}\right. \nonumber \\
        &\quad\quad\left. -\frac{\log(x-1)-\log x}{\varphi(m)}+\left(\frac{\log^2 M+4\log M+8}{\sqrt{M}}\right)\left(\frac{1}{\pi}-2.096a(m)\right) \right. \label{eq:absoluteValue}\\
        &\quad\quad\left. -\frac{1}{\pi}\left(\frac{\log^2 2+4\log 2+8}{\sqrt{2}}\right)+2.096a(m)\left(\frac{\log^2 x+4\log x+8}{\sqrt{x}}\right) \right|. \nonumber
\end{align}
Since $
    \frac{1}{M-1}-\log(M-1)<0,
$ 
$-1-\log(x-1)+\log x <0$ for all $x$ under consideration, 
$\varphi(m) \geq 2$, the function
$
    \left(\log^2 y+4\log y+8\right)/\sqrt{y}
$
is decreasing for all $y>0$ and $a(m) \geq 0.129>0$, we have
\begin{align*}
    &\frac{\varphi(m)-1}{\varphi(m)}\left(\frac{1}{M-1}-\log(M-1)\right)<0, \quad -1+\frac{-\log(x-1)+\log x}{\varphi(m)}<0 \\
    & \frac{1}{\pi}\left(\frac{\log^2 M+4\log M+8}{\sqrt{M}}-\frac{\log^2 2+4\log 2+8}{\sqrt{2}}\right)<0 \\
    &2.096a(m)\left(\frac{\log^2 x+4\log x+8}{\sqrt{x}}-\frac{\log^2 M+4\log M+8}{\sqrt{M}}\right)<0.
\end{align*}
Hence, the sum inside the absolute values in \eqref{eq:absoluteValue} is negative. Thus, when the positive terms are removed, with an exception of the term $2.096a(m)-1/\pi$ that can be positive or negative, the absolute value increases. By these, the right-hand side of \eqref{eq:absoluteValue} is
\begin{align}
        <\frac{0.092\sqrt{x}\log^2 x}{x-1} +\log M +\frac{3}{2}+\frac{1}{\pi}\left(\frac{\log^2 2+4\log 2+8}{\sqrt{2}}\right) \nonumber\\
        +\left(2.096a(m)-\frac{1}{\pi}\right)\left(\frac{\log^2 M+4\log M+8}{\sqrt{M}}\right). \label{eq:EstMainTerm} 
\end{align}
We are almost ready, we still want to simplify the last line in estimate \eqref{eq:EstMainTerm} a little bit. Let us consider the last line in \eqref{eq:EstMainTerm} in different cases depending on the size of $m$.
In each of the cases we apply the facts that the functions
    $\left(\log^2 y+4\log y+8\right)/\sqrt{y}$ and $\sqrt{y}\log^2 y/(y-1)$
are decreasing for all $y>0$. 

In the case $3 \leq m \leq 432$, we apply the facts $2.096a(m)-1/\pi<0$, $M = \max\{m,\sqrt{1865}\}$ and $x \geq \max\{1865, m^2\}$. This gives upper bound $3.883+\log\sqrt{1865}$ when $m \leq 43$ and $3.994+\log m$ when $44\leq m\leq 432$. Further, in the case $m \geq 433$ we use $M \geq m\geq 433$ and that by Corollary \ref{corollary:varLower} and assumptions we have $x \geq m^{\varphi(m)} \geq 433^{433^{0.7}}$. We get an upper bound $5.754+\log m$. Hence, an upper bound $6.550+\log m$ holds in every case. 
\end{proof}

\begin{remark}
In the previous proof, we could combine estimates coming from Theorem \ref{psi}, inequality \eqref{eq:psiEctraTerm} and the fact that $x \geq 2$, to obtain a bound
\begin{equation}
\label{eq:ySecond}
    \frac{y}{\varphi(m)}+7.171\sqrt{y}\log^2 y+\sqrt{x}\log(m)\left(\frac{\log y}{2\pi}+2\right)
\end{equation}
and use it instead of the right-hand side of estimate \eqref{eq:y}. However, the number $y$ must be very large (e.g. clearly large than $10^{6}$) before estimate \eqref{eq:ySecond} is better than estimate on the right-hand side of \eqref{eq:y}. Since in \eqref{eq:y} we want to also consider small values $y$ and the upper bound in this consideration is $\max\{\sqrt{1865},m\}$, we use the right-hand side of estimate \eqref{eq:y}.
\end{remark}

\begin{lemma}
\label{lemma:Log}
Let $a$ be defined as in Theorem \ref{psi} and $m$ and $x$ as in Lemma \ref{lemma:estimateSumPNLog}. Assume also RH and that the $m$th cyclotomic Dedekind zeta function satisfies the GRH. Further, let $n \geq x$ be a positive integer.
Then we have
\begin{multline*}
    \left|\log\left(\prod_{\substack{p \leq x \\ p \equiv a \pmod m}} \left|n!\right|_p\right)+\frac{n\log x}{\varphi(m)}\right|< (6.550+\log m)n+\frac{x\log n}{\varphi(m)\log x}+\frac{x}{\varphi(m)} \\
    +\left(2.539+\frac{5.440}{\varphi(m)}\right)\frac{n}{\log n}.
\end{multline*}
\end{lemma}

\begin{proof}
We follow the same steps as in the proof of Lemma 9 in \cite{EHMS2019}. First, we modify the first term on the left-hand side of the claim and then derive the right-hand side using Lemma \ref{corollary:theta}. 

First, we write the first term on the left-hand side as
\begin{equation}
\label{eq:asSum}
    \log\left(\prod_{\substack{p \leq x \\p \equiv a \pmod m}} \left|n!\right|_p\right)=\sum_{\substack{p \leq x \\ p \equiv a \pmod m}} \log |n!|_p.
\end{equation}
As in \cite[Proof of Lemma 9]{EHMS2019}, by a well known fact we have
\begin{equation*}
 p^{-\frac{n}{p-1}} \leq \left|n!\right|_p \leq p^{-\frac{n}{p-1}+\frac{\log n}{\log p}+1}
\end{equation*}
and hence we can estimate the right-hand side of equation \eqref{eq:asSum} with sums
\begin{equation}
\label{eq:sumNThree}
    -\sum_{\substack{ p \leq x \\p \equiv a \pmod m}} \frac{n}{p-1}\log p \quad\text{and} \sum_{\substack{ p \leq x \\p \equiv a \pmod m}} \left(\log n+\log p\right).
\end{equation}
We consider these sums separately.

First we notice that the first sum in \eqref{eq:sumNThree} can be estimated with Lemma \ref{lemma:estimateSumPNLog}. Hence we can move to the rest of the terms in \eqref{eq:sumNThree}.

By Lemma \ref{lemma:estPi} and Lemma \ref{corollary:theta2} we have
\begin{multline}
\label{eq:logSumSecondTerm}
    \left|\sum_{\substack{ p \leq x \\p \equiv a \pmod m}} \left(\log n+\log p\right)\right|<\log n\left(\frac{\textrm{Li}(x)}{\varphi(m)}+\frac{\sqrt{x}\log x}{8\pi}+ \sqrt{x}\left(\frac{\log m}{2\pi}+\frac{3}{\varphi(m)}+0.341\right) \right) \\
    +\frac{x}{\varphi(m)}+0.092\sqrt{x}\log^2 x.
\end{multline}
Further, by \cite[Lemma 5.9]{BGOR2018} we have
\begin{equation}
\label{eq:li}
    \textrm{Li}(x)<\frac{x}{\log x}+\frac{3x}{2\log^2 x}
\end{equation}
for all $x \geq 1865$. Hence we have estimated the second sum in \eqref{eq:sumNThree}. 

Let us now combine our estimates. Using Lemma \ref{lemma:estimateSumPNLog} and estimates \eqref{eq:logSumSecondTerm} and \eqref{eq:li} for the sums in \eqref{eq:sumNThree} we have
\begin{align*}
    &\left|\log\left(\prod_{\substack{p \leq x \\ p \equiv a \pmod m}} \left|n!\right|_p\right)+\frac{n\log x}{\varphi(m)}\right|<(6.550+\log m)n+\frac{x\log n}{\varphi(m)\log x}+\frac{x}{\varphi(m)}+\frac{3x\log n}{2\varphi(m)\log^2 x}  \\
    &\quad+\frac{\sqrt{x}\log x \cdot \log n}{8\pi}+0.092\sqrt{x}\log^2 x+\sqrt{x}\log n\left(\frac{\log m}{2\pi}+\frac{3}{\varphi(m)}+0.341\right).
\end{align*}
Since $x \leq n$ and the function $x/\log^2 x$ is increasing for all $x$ under consideration, the right-hand side is
\begin{multline*}
   \leq (6.550+\log m)n+\frac{x\log n}{\varphi(m)\log x}+\frac{x}{\varphi(m)} \\
    +\frac{n}{\log n}\left(\frac{3}{2\varphi(m)}+\left(\frac{1}{8\pi}+\frac{1}{4\pi}+0.092\right)\frac{\log^3 n}{\sqrt{n}}+\left(\frac{3}{\varphi(m)}+0.341\right)\frac{\log^2 n}{\sqrt{n}}\right).
\end{multline*}
Since functions $\log^3 n/\sqrt{n}$, $\log^2 n/\sqrt{n}$ and $\log^2 n/n$ are decreasing for all $n \ge 1865$, the claim follows.
\end{proof}

\begin{remark}
\label{remark:Log}
Since $|q|_p=1$ for all primes $p>q$, for all $x>n$ we have
\begin{equation*}
   \log\left(\prod_{\substack{p \in \mathbb{P} \\ p \equiv a \pmod m}} \left|n!\right|_p\right)= \log\left(\prod_{\substack{p \leq x \\ p \equiv a \pmod m}} \left|n!\right|_p\right)=\log\left(\prod_{\substack{p \leq n \\ p \equiv a \pmod m}} \left|n!\right|_p\right).
\end{equation*}
\end{remark}

\subsection{Non-empty sets}
\label{sec:Existence}
In this section, we prove that certain sets are non-empty. These results are applied in Sections \ref{sec:FirstLower} and \ref{sec:lower2}. First we mention a useful lemma and then we prove the result.

\begin{lemma}\cite[Theorem 1.1, Table 1]{DGM2019}
\label{lemma:primeExists}
Assume GRH, let $m \geq 3$ and $a$ be integers such that $(a,m)=1$. Further, assume that 
$
    h \geq \varphi(m)\left(\log x/2+\log m+12\right)\sqrt{x}
$
and $x \geq \left(23\varphi(m)\log m \right)^2$. Then there is a prime $p$ which is congruent to $a$ modulo $m$ with $\left|p-x\right|<h$.
\end{lemma}

\begin{lemma}
\label{lemma:exists}
Let us use the same definitions and assumptions as in Theorem \ref{thm:lowerBound}, let $R_1$ correspond to $R$ in Theorem \ref{thm:dmGene} and let $x_1 \geq m^{\varphi(m)}$. Then there is at least one prime in the set $R_1\cap [2,x_1]$.

Further, let us use the same definitions and assumptions as in Theorem \ref{thm:lowerBoundMore}, let $R_2$ correspond to $R$ in Theorem \ref{thm:lowerBoundMore} and let $\log x_2 \geq m^{\varphi(m)}$. Then there is at least one prime in the set $R_2 \cap (\log x_2, x_2]$.
\end{lemma}

\begin{proof}
The idea is that we first estimate the sizes of the terms $x_1$ and $x_2$ and then apply Lemma \ref{lemma:primeExists}.
First, it is easy to check that there is at least one prime in $R_1 \cap \left[2,x_1\right]$ when $m \in [3,60]$. Hence, by Corollary \ref{corollary:varLower} we can assume that
\begin{equation*}
    \frac{x_1}{2} \geq \frac{m^{\varphi(m)}}{2} \geq 0.5\cdot61^{61^{0.7}}>10^{31} \quad\text{and}\quad \frac{x_2}{2}\geq \frac{m^{\varphi(m)}e^{m^{\varphi(m)}}}{2\log\left(m^{\varphi(m)}e^{m^{\varphi(m)}}\right)}>3256.
\end{equation*} 
Even more, we clearly have $x_2/2>0.4e^{m^{\varphi(m)}}>\left(23\varphi(m)\log m\right)^2$ for all $m \geq 3$ and $\varphi(m) \geq 2$ and
\begin{equation*}
    \log x_2+ 2\left(\frac{\log x_2}{2}+\log 3+12\right)\sqrt{x_2}<\frac{x_2}{2}.
\end{equation*}
Hence, by Lemma \ref{lemma:primeExists} there is at least one prime in $R_2\cap(\log x_2, x_2]$ when $m=3$ and we can assume $m \geq 4$ and $x_2/2 >10^6$ in this case. Further, since $x_2/2>m^{\varphi(m)}/2$ and $\log x_2<x_2/100$, in order to prove the claim, by Lemma \ref{lemma:primeExists} it is sufficient to show that
\begin{equation}
\label{eq:ToProveBound}
    \varphi(m)\left(\frac{\log x}{2}+\log m+12\right)\sqrt{x}<0.49x \quad\text{and}\quad 0.5x_1 \geq \left(23\log (0.5x_1)\right)^2
\end{equation}
for $x_1>10^{31}$, $x\geq \max\{m^{\varphi(m)}/2,10^6\}$ where $m \geq 4$.

The last inequality in \eqref{eq:ToProveBound} is clearly true and we can concentrate on to prove the first inequality in \eqref{eq:ToProveBound}. Since $\varphi(m)\log x\leq \log(2x)\log x/\log m<\sqrt{x}/6$ and $\log(2x)<\sqrt{x}/60$ for all $x$ under consideration, we have
\begin{equation*}
    \varphi(m)\left(\frac{\log x}{2}+\log m+12\right)\sqrt{x}<\frac{x}{12}+\frac{x}{60}+\frac{x}{60\log 4}<0.49x,
\end{equation*}
which proves the claim.
\end{proof}

\section{Proof of Theorem \ref{thm:limSupCond}}
\label{sec:proof1}
Let us assume contrary, i.e. that for all $p \in R$ we have $\Lambda_p =0$. As in Section \ref{sec:outline}, we obtain
\begin{equation*}
    1 \leq (k+1)\max_{0 \leq i\leq k} \{\left|\lambda_i\right|\} \max_{0 \leq i \leq k} \{\left|B_{n,\mu,i}(1)\right|\}\cdot \max_{1 \leq i \leq k} \left\{ \prod_{p \in R}\left|S_{n,\mu,i}(1)\right|_p\right\},
\end{equation*}
where $\mu$ satisfies property \eqref{eq:defMu} for $n$ instead of $n+1$. It should be noted that for every $n \in \mathbb{Z_+}$ we can find such number $\mu$ and for all of these choices we have $0 \leq \mu \leq k$. Hence for such $\mu$, by Lemma \ref{lemma:BSEstimates} the right-hand side is
\begin{equation*}
    \leq (k+1)\max_{0 \leq i\leq k} \{\left|\lambda_i\right|\}\left(\max_{1\leq j \leq k}\{|\alpha_j|\}\right)^{k-1}c_2(\overline{\alpha}; R)^n(kn+k)!(kn+k)\prod_{p \in R}\left|(kn)!n!\right|_p,
\end{equation*}
where $c_2(\overline{\alpha}; R)$ is defined as in \eqref{def:c1}. Thus, if condition \eqref{eq:limSupCond} holds, then the right-hand side of the previous inequality is smaller than one when $n$ is large enough. This is a contradiction and hence we cannot have $\Lambda_p =0$ for all $p \in R$.

\section{Proof of Theorem \ref{thm:dmGene}}
\label{sec:proof2}

First of all, by Theorem \ref{thm:limSupCond} and its proof it is sufficient to show that the limit superior of formula
\begin{equation}
\label{eq:toMinusInf}
    n\log c_2(\overline{\alpha}; R)+\log(kn+k)+\log((kn+k)!)+\log\left(\prod_{p \in R} \left|(nk)! \cdot n!\right|_p\right)
\end{equation}
is minus infinity.

First we consider the second and third term in sum \eqref{eq:toMinusInf}. The second term can be estimated with
\begin{equation}
\label{eq:Logknk}
    \log(kn+k)=\log(n+1)+\log k< \log n+\log k+\frac{1}{n},
\end{equation}
for all $n >0$. By the previous estimate and by Stirling's formula (see e.g. \cite[formula 6.1.38]{AS1964}) we have
\begin{align*}
    \log((kn+k)!)&=\left(kn+k\right)\log(kn+k)-(kn+k)+O(\log n)=kn\log n+kn\log \frac{k}{e}+O(\log n).
\end{align*}

Using the previous bounds, Lemma \ref{lemma:Log} and Remark \ref{remark:Log}, sum \eqref{eq:toMinusInf} is
\begin{align*}
    &\leq n\log c_2(\overline{\alpha}; R)+kn\log n+kn\left(\log k-1\right)+O(\log n)\\
    &\quad+\sum_{j=1}^{k\varphi(m)/(k+1)}\left(\log\left(\prod_{p \in \overline{h}_j \cap \mathbb{P}} \left|(nk)!\right|_p\right)+\log\left(\prod_{p \in \overline{h}_j \cap \mathbb{P}} \left|n!\right|_p\right)\right) \\
    &\quad<n\left(\log c_2(\overline{\alpha}; R)+k\left(\log k+(\log m+6.550)\varphi(m)+1\right)\right)+O\left(\frac{n}{\log n}\right).
\end{align*}
Using upper bound \eqref{eq:cUpper} for the number $c_2(\overline{\alpha}; R)$, the right-hand side of the previous inequality goes to $-\infty$ when $n$ goes to infinity. By Theorem \ref{thm:limSupCond}  the claim follows.

\section{Preliminaries and proofs for Theorem \ref{thm:lowerBound} and Corollary \ref{corollary:lowerBound}}
\label{sec:FirstLower}
In this section, we prove Theorem \ref{thm:lowerBound} and Corollary \ref{corollary:lowerBound}. The section is divided into two parts: First we derive the contradiction that the right-hand side of inequality \eqref{eq:Product} is smaller than one (see Section \ref{sec:outline}) as well as estimates for the number $n$ which gives the contradiction. After that, we are ready to prove Theorem \ref{thm:lowerBound} and Corollary \ref{corollary:lowerBound}.

\subsection{Contradiction and estimates for the number $n$}
\label{sec:Contra}

In this section, we show that the right-hand side of inequality \eqref{eq:Product} is smaller than one for a certain $n$ and we give a
convenient bound for this $n$.

First, we define
\begin{align*}
    N_1(a, R)&:=a\left(\log c_2(\overline{\alpha}; R)+k\left(\log k+\left(\log m+6.55\right)\varphi(m)+1\right)\right)+\left(2.539+\frac{5.440}{\varphi(m)}\right)\frac{k\varphi(m)a}{\log a} \\
   &\quad+2.5\log a+\log(k+1)+2.5\log k+\log H+ (k-1)\log\left(\max_{1\leq j \leq k}\{|\alpha_j|\}\right)\\
   &\quad+\log\sqrt{2\pi}+1+k\left(\left(\log m+6.55\right)\varphi(m)+2\right)+\left(2.539+\frac{5.440}{\varphi(m)}\right)\frac{k\varphi(m)}{\log a}+\frac{19}{12a}.
\end{align*}

Now we prove the contradiction keeping in mind that the term $c_2(\overline{\alpha}; R)$ is defined in \eqref{def:c1}:
\begin{lemma}
\label{lemma:contra}
Let $k,m,\lambda_0,\ldots, \lambda_k$ and $R$ be defined as in Theorem \ref{thm:dmGene}. Assume RH and that the $m$th cyclotomic Dedekind zeta function satisfies the GRH. Further, let $\alpha_1, \ldots, \alpha_k$ and $R'$ be defined as in Theorem \ref{thm:lowerBound}. Suppose that
\begin{equation}
\label{eq:HBound}
    \frac{\log H}{D(m,k,\overline{\alpha}; R')} \geq \max \left\{1866, m^{\varphi(m)}+1, \max\limits_{1 \leq j \leq k} \{|\alpha_j|\}+1 \right\},
\end{equation}
where $D(m,k,\overline{\alpha}; R')$ is given as in \eqref{eq:DefD},
\begin{equation}
\label{eq:nBound}
    n:=\max\left\{ a \in \mathbb{Z}: N_1\left(a, R'\right) \geq 0\right\}
\end{equation}
and $\mu$ satisfies property \eqref{eq:defMu}.

Then, we have 
\begin{equation*}
   (k+1)\max_{0 \leq i\leq k} \{\left|\lambda_i\right|\} \max_{0 \leq i \leq k} \{\left|B_{n+1,\mu,i}(1)\right|\}\cdot \max_{1 \leq i \leq k}\left\{ \prod_{p \in R\cap [2,k(n+2)]}\left|S_{n+1,\mu,i}(1)\right|_p\right\}<1.
\end{equation*}
\end{lemma}

\begin{proof}
We would like to estimate the left-hand side of the claim. Hence, we first notice that due to assumptions \eqref{eq:nBound} and \eqref{eq:HBound} we must have
\begin{equation*}
    n> \frac{\log H}{D(m,k,\overline{\alpha}; R')}-1\geq \max\{1865, m^{\varphi(m)}, \max_{1 \leq j \leq k} \{|\alpha_j|\}\}.
\end{equation*}
Thus we can apply Lemma \ref{lemma:Log} in our estimates and $R' \subseteq R\cap [2, k(n+2)]$.

Now we estimate the left-hand side of the claim. First, we notice that by Lemma \ref{lemma:exists} there is at least one prime in the set $R\cap [2, k(n+2)]$. As in the proof of Theorem \ref{thm:limSupCond}, using Lemma \ref{lemma:BSEstimates} and keeping in mind $\mu \in \{0,1,\ldots, k\}$, the left-hand side of the claim is
\begin{align*}
    & \leq (k+1)H\left(\max_{1\leq j \leq k}\{|\alpha_j|\}\right)^{k-1}c_2(\overline{\alpha}; R')^{n+1}\left(k(n+1)+\mu\right)!\cdot\left(k(n+1)+k\right) \\
    &\quad \left(\prod_{p \in R\cap [2,k(n+2)]}\left|(k(n+1)+\mu)!\right|_p\right)\cdot \left(\prod_{p \in R\cap [2,k(n+2)]}\left|(n+1)!\right|_p\right).
\end{align*}
Similarly as in the proof of Theorem \ref{thm:dmGene} and especially in estimate \eqref{eq:Logknk} and by Stirling's formula (see e.g. \cite[formula 6.1.38]{AS1964}), the logarithm of the right-hand side can be estimated as
\begin{align*}
    <\log(k+1)+\log H+ (k-1)\log\left(\max_{1\leq j \leq k}\{|\alpha_j|\}\right)+(n+1)\log c_2(\overline{\alpha}; R')+\log k+\log(n+1) \\
   +\frac{1}{n+1}+ \left(k(n+1)+\mu+0.5\right)\log(k(n+1)+\mu)-k(n+1)-\mu+\log\sqrt{2\pi}+\frac{1}{12(k(n+1)+\mu)}\\ 
    +\log\left(\prod_{p \in R\cap[2,k(n+2)]}\left|(k(n+1)+\mu)!\right|_p\right)+\log\left(\prod_{p \in R\cap[2,k(n+2)]} \left|(n+1)!\right|_p\right).
\end{align*}
Our goal is to show that the previous estimate is at most zero. Since $R\cap[2,k(n+2)]$ contains primes in arithmetic progressions up to height $k(n+2)$ and $0 \leq \mu \leq k$, by Lemma \ref{lemma:Log} and Remark \ref{remark:Log} this is
\begin{align*}
    &<\log(k+1)+\log H+ (k-1)\log\left(\max_{1\leq j \leq k}\{|\alpha_j|\}\right)+(n+1)\log c_2(\overline{\alpha}; R')+\log k+\log (n+1) \\
    &\quad+\frac{13}{12(n+1)}+\left(k(n+1)+\mu+0.5\right)\log(k(n+1)+\mu)-k(n+1)-\mu+\log\sqrt{2\pi}\\ 
   &\quad+\frac{k\varphi(m)}{k+1}\left( -\frac{(k(n+1)+\mu)\log (k(n+1)+\mu)}{\varphi(m)}+\left(\log m+6.55+\frac{2}{\varphi(m)}\right)(k(n+1)+\mu)\right.\\
   &\quad\left.+\left(2.539+\frac{5.440}{\varphi(m)}\right)\frac{k(n+1)+\mu}{\log (k(n+1)+\mu)}-\frac{(n+1)\log( n+1)}{\varphi(m)}\right.\\
   &\quad\left.+\left(\log m+6.55+\frac{2}{\varphi(m)}\right)(n+1)+\left(2.539+\frac{5.440}{\varphi(m)}\right)\frac{(n+1)}{\log (n+1)}\right) \\
   &\quad < (n+1)\left(\log c_2(\overline{\alpha}; R')+k\left(\log k+\left(\log m+6.55\right)\varphi(m)+1\right)\right)\\
   &\quad\quad+\left(2.539+\frac{5.440}{\varphi(m)}\right)\frac{k\varphi(m)(n+1)}{\log(n+1)} +2.5\log(n+1)+\log(k+1)+2.5\log k+\log H\\
   &\quad\quad+ (k-1)\log\left(\max_{1\leq j \leq k}\{|\alpha_j|\}\right)+\log\sqrt{2\pi}+1+k\left(\left(\log m+6.55\right)\varphi(m)+2\right) \\
   &\quad\quad+\left(2.539+\frac{5.440}{\varphi(m)}\right)\frac{k\varphi(m)}{\log (n+1)}+\frac{19}{12(n+1)} \\
   &\quad=N_1(n+1, R').
\end{align*} 
Because of assumption \eqref{eq:nBound}, the right-hand side is smaller than zero.
\end{proof}

\begin{remark}
The assumption \eqref{eq:cUpper} is essential in order to find number $n$ such that $N_1(n+1, R')<0$.
\end{remark}

In the previous lemma, the number $n$ is given a quite complicated way. In order to apply the result, we would like to give a simplified upper bound for this quantity. We have included more assumptions for the number $H$ in this simplified version to obtain a simpler result. 

\begin{lemma}
\label{lemma:UpperBoundsN} 
Let us use the same definitions and assumptions as in Lemma \ref{lemma:contra} including that $n$ is defined as in \eqref{eq:nBound}. 
Assume also that $H$ satisfies the bounds given in \eqref{eq:BoundH}.

Then we have
\begin{equation*}
    n< \frac{\left(2.539k\varphi(m)+5.724k+0.054\right)n/\log n+\log H}{D(m,k,\overline{\alpha}; R')}
\end{equation*}
and
\begin{equation*}
   n< \frac{2\log H}{D(m,k,\overline{\alpha}; R')}.
\end{equation*}
\end{lemma}

\begin{proof}
The idea is to prove bounds for the number $n$ using the term $N_1(n, R')$: By assumption \eqref{eq:nBound} we have $0 \leq N_1(n, R')$. We derive an upper found of form
$$
    -D(m,k,\overline{\alpha}; R') n+\frac{A(m,k)n}{\log n}+\log H,
$$
where $A(m,k)$ depends only terms $m$ and $k$, for the term $N_1(n, R')$. Using this upper bound and keeping the inequality $0 \leq N_1(n, R')$ in mind, we find the wanted upper bounds.

Let us now derive the upper bound for the term $N_1(n, R')$. As in the first paragraph of the proof of Lemma \ref{lemma:contra}, we can deduce that $n>m^{\varphi(m)}$ and $n> \max_{1\leq j \leq k}\{|\alpha_j|\}$ and we also have $\log k<\log (k+1)<k$ for all $k \geq 1$. Hence, the term $N_1(n, R')$ can be estimated as
\begin{align*}
    &N_1(n, R')<-nD(m,k,\overline{\alpha}; R')+\frac{n}{\log n}\left(\vphantom{\frac{2.5\log^2 n}{n}} \left(2.539+\frac{5.440}{\varphi(m)}\right)k\varphi(m)+\frac{(k+1.5)\log^2 n}{n}+\left(\log m+6.55\right)\cdot\right. \\
   &\quad\left.\cdot\frac{k\log^2 n}{n\log m}+\frac{\left(5.5k+\log\sqrt{2\pi}+1\right)\log n}{n} +\left(2.539+\frac{5.440}{\varphi(m)}\right)\frac{k\log n}{n\log m} +\frac{19\log n}{12n^2}\right)+\log H.
\end{align*}
The coefficient for $n/\log n$ follows when we set $n=1865$, $\log m=\log 3$ and $\varphi(m)=2$ when $\varphi(m)$ is in the denominator. Hence, we have found the wanted upper bound for $N_1(n, R')$ and the first bound follows.

Let us now move to the second bound. We have obtained
\begin{equation}
\label{eq:A1Positive}
    -\left(-D(m,k,\overline{\alpha}; R')+\frac{2.539k\varphi(m)+5.724k+0.054}{\log n}\right) n<\log H.
\end{equation}
Due to the assumption \eqref{eq:BoundH}, the left-hand side is positive and
$$
\log n\geq \frac{2\left(2.539k\varphi(m)+5.724k+0.054\right)}{D(m,k,\overline{\alpha}; R')}.
$$ 
The second result follows using the previous bound for $\log n$.
\end{proof}

\begin{remark}
It is actually sufficient to assume that
$$
 \frac{\log H}{D(m,k,\overline{\alpha}; R')} \geq  e^{\frac{2.539k\varphi(m)+5.724k+0.054}{D(m,k,\overline{\alpha}; R')}}+1
$$
in order the left-hand side of \eqref{eq:A1Positive} to be positive. However, an exact lower bound would lead the left-hand side to be arbitrary close to zero. To avoid this problem and to make computations easier, a larger lower bound was selected.
\end{remark}

\subsection{Proofs for Theorem \ref{thm:lowerBound} and Corollary \ref{corollary:lowerBound}}
\label{sec:proof3}

In this section, we prove $p$-adic lower bounds for linear forms in the Euler's factorial series.
First, we prove a useful lemma regarding the lower bound for the term $\log H$. In the proof, $W(x)$ denotes the Lambert $W$ function meaning that $W(x)e^{W(x)}=x$ and $W(x)$ corresponds to the principal value.

\begin{lemma}
\label{lemma:xex}
Assume that $\log H \geq xe^{y/x}+x$, where $H$, $x$ and $y \geq 1$ are positive real numbers. Then
$$
\frac{y}{2\log\log H}\leq x \leq \frac{\log H}{2}.
$$
\end{lemma}

\begin{proof}
Let us prove the claim by assuming the contrary: $x < y/(2\log\log H)$ or $x > \log H/2$.

Let us first consider the first case. First we notice that the function $xe^{y/x}+x$ obtains its minimum at $x=y/\left(W(1/e)+1\right)$, where $W$ is the Lambert $W$ function. Hence, we have
\begin{equation}
\label{eq:HLowerA1}
\log H \geq xe^{y/x}+x\geq \frac{y}{W\left(1/e\right)+1}\left(e^{W(1/e)+1}+1\right)>3.591y.
\end{equation}
Thus, we also have 
$$
\frac{y}{2\log\log H}<\frac{y}{2\log(3.591y)}<y/\left(W(1/e)+1\right) \approx 0.782y,
$$
Also, since the function $xe^{y/x}+x$ is decreasing for all $x<y/\left(W(1/e)+1\right)$, we have
$$
xe^{y/x}+x > \frac{ye^{2\log\log H}}{2\log\log H}+\frac{y}{2\log\log H}> \frac{\left(\log H\right)^{2}}{\log\log H}>\log H,
$$
for all $H>e$. This is a contradiction.

In the second case we have
$$
xe^{y/x}+x>\frac{\log H\cdot e^0}{2}+\frac{\log H}{2}=\log H,
$$
which is again a contradiction. Hence, the claim holds.
\end{proof}

\begin{remark}
It is possible to prove a little bit sharper bounds in Lemma \ref{lemma:xex} formulating the bounds with the Lambert $W$ function instead of logarithms. However, in order to keep the results relatively simple, logarithms where decided to be used.
\end{remark}

Now we are ready for the proof of Theorem \ref{thm:lowerBound}:

\begin{proof}[Proof of Theorem \ref{thm:lowerBound}]
Let $n$ and $\mu$ be is in Lemma \ref{lemma:contra}. First of all, by Lemma \ref{lemma:exists}, there is at least one prime number in the set $R\cap [2,k(n+2)]$. Further, because of Lemma \ref{lemma:contra}, there exists a number $n$ such that the right-hand side of estimate \eqref{eq:Product} is smaller than one. On the other hand, in estimate \eqref{eq:Product} we noticed that the same term is at least one if there is no number $p \in R\cap [2,k(n+2)]$ such that
\begin{equation}
\label{eq:BLarge}
\left|B_{n+1,\mu,0}(1)\Lambda_p\right|_p > \left|\sum_{j=1}^k \lambda_jS_{n+1,\mu,j}(1)\right|_p. 
\end{equation}
This is a contradiction and hence there must be a prime number $p \in R\cap [2,k(n+2)]$ such that inequality \eqref{eq:BLarge} holds.
We consider this prime number $p$. First we prove that $p$ is in the set \eqref{eq:interval} and then that the wanted lower bound is satisfied. 

Trivially we must have $p \geq 2$. Further, by Lemma \ref{lemma:UpperBoundsN}, we have
\begin{equation*}
    k(n+2)<\frac{2k\log H}{D(m,k,\overline{\alpha}; R')}+2k.
\end{equation*}
Hence, number $p$ is in the wanted set and we can move on to find a $p$-adic lower bound for the linear form. 

We only have to find a lower bound for the term $\left|\Lambda_p\right|_p$. By estimate \eqref{eq:TnEnough} it is sufficient to find an upper bound for the term $\left|T(n+1, \mu)\right|$, so we will consider this term.

By definition \eqref{eq:defT} of the term $T(n+1, \mu)$, Lemma \ref{lemma:BSEstimates} and Stirling's formula (see e.g. \cite[formula 6.1.38]{AS1964}) we have
\begin{multline*}
    \left|T(n+1, \mu)\right| \leq (k+1) H \max_{1 \leq j \leq k} \left|B_{n+1,\mu, j}(1)\right|\leq (k+1)H\sqrt{2\pi} \cdot e^{\frac{1}{12k(n+1)}-k}\left(\max_{1 \leq j \leq k} \left\{|\alpha_j|\right\}\right)^{k-1}\cdot  \\
   \cdot\left(k(n+2)\right)^{k+1.5} e^{k(n+1)\left(\log 2+\log\left(k(n+2)\right)-1+\log \left(\max_{1 \leq j \leq k} \left\{|\alpha_j|\right\}\right)\right)}.
\end{multline*}
Applying the facts $\log(a+1)<\log a+1/a$, $\log (k+1)<k$ and $k \geq 1$ and assumption \eqref{eq:BoundH} meaning that $n> \max_{1\leq j \leq k}\{|\alpha_j|\}$, the logarithm of the right-hand side is
\begin{multline}
\label{eq:almostUpper}
    <\log H+kn\log n+n\left(k\log \left(\max_{1 \leq j \leq k} \left\{|\alpha_j|\right\}\right)+k\left(k+\log 2-1\right)\right) \\
     +\log n\left(4k+0.5+\frac{(2k+2.5+\log{2}) k+\log\sqrt{2\pi}}{\log n}+\frac{4k+37/12}{n\log n}\right).
\end{multline}
By Lemma \ref{lemma:UpperBoundsN}, we can estimate
\begin{equation}
\label{eq:estnLogn}
    n\log n< \frac{\left(2.539k\varphi(m)+5.724k+0.054\right)n+\log H\cdot \log n}{D(m,k,\overline{\alpha}; R')}
\end{equation}
and use the second estimate proved for the number $n$ in Lemma \ref{lemma:UpperBoundsN}.
Hence, estimate \eqref{eq:almostUpper} is
\begin{align}
    &<\log H+k\log H\cdot \frac{\log\log H+\log 2-\log D}{D}+\frac{kA(m,k)\log H}{D^2}+\frac{kA(m,k)^2n/\log n}{D^2} \nonumber \\
    &\quad+\left(\frac{\log H}{D}+\frac{A(m,k)n/\log n}{D}\right)\left(k\log \left(\max_{1 \leq j \leq k} \left\{|\alpha_j|\right\}\right)+k\left(k+\log 2-1\right)\right) \nonumber\\
    &\quad+\log\left(\frac{2\log H}{D}\right)\cdot\left(4k+0.5+\frac{(2k+2.5+\log2) k+\log\sqrt{2\pi}}{\log n}+\frac{4k+37/12}{n\log n}\right)\nonumber \\
    &<\quad \log H\cdot\left(\frac{k}{D}\log\log H+1+\frac{k\left(\log 2-\log D\right)}{D}+\frac{kA(m,k)}{D^2}\right. \label{eq:DSum} \\
    &\quad\quad\left.+\frac{k\log \left(\max_{1 \leq j \leq k} \left\{|\alpha_j|\right\}\right)+k\left(k+\log 2-1\right)}{D}+\frac{2kA(m,k)^2}{D^3\left(\log\log H+\log (2/D)\right)} \right. \nonumber \\
    &\quad\quad\left. +\frac{2A(m,k)\left(k\log \left(\max_{1 \leq j \leq k} \left\{|\alpha_j|\right\}\right)+k\left(k+\log 2-1\right)\right)}{D^2\left(\log\log H+\log (2/D)\right)}\right. \nonumber\\
    &\quad\quad\left. +\frac{\log\left(2\log H/D\right)}{\log H}\cdot\left(4k+0.5+\frac{(2k+2.5+\log 2) k+\log\sqrt{2\pi}}{\log n}+\frac{4k+37/12}{n\log n}\right)\right) \nonumber,
\end{align}
where $D:=D(m,k,\overline{\alpha}; R')$ and $A(m,k):=\left(2.539k\varphi(m)+5.724k+0.054\right)$.
Next we simplify the previous estimate. Our goal is to write it in a form
$
    A\log H\cdot\log\log H+B,
$
where $A$ and $B$ do not depend on number $H$ but may depend on other terms such as number $m$.

Since by assumption \eqref{eq:BoundH} we have
$$
\log\log H-\log D>\frac{2A(m,k)}{D} \quad\text{and}\quad \log\log H-\log D>\max\left\{1,\log\left(\max_{i\leq j\leq k}\left\{|\alpha_j|\right\}\right)\right\},
$$
using also Lemma \ref{lemma:xex} the terms inside the brackets in the right-hand side of inequality \eqref{eq:DSum} are
\begin{align*}
    &< \frac{k}{D}\log\log H+1+\frac{k\left(\log 2-\log D\right)}{D}+\frac{kA(m,k)}{D^2}+\frac{k\left(\log\log H-\log D\right)}{D} \\
    &\quad+\frac{k\left(k+\log 2-1\right)}{D}+\frac{kA(m,k)}{D^2}+\frac{2A(m,k)k}{D^2}+\frac{k\left(k+\log 2-1\right)}{D} \\
    &\quad +\frac{\log\left(2\log H\cdot\log\log H/A(m,k)\right)}{\log H}\cdot\left(4k+0.5+\frac{(2k+2.5+\log 2) k+\log\sqrt{2\pi}}{\log n}+\frac{4k+37/12}{n\log n}\right) \\
    &\quad<\frac{2k}{D}\log\log H+1-\frac{2k\log D}{D}+\frac{2k^2+3k\log 2-2k}{D}+\frac{4kA(m,k)}{D^2} \\
    &\quad\quad +\frac{\log\log H+\log \log H+\log(2/A(m,k))}{\log H}\cdot\left(4k+0.5+\frac{(2k+2.5+\log 2) k+\log\sqrt{2\pi}}{\log n}\right. \\
    &\quad\quad\left.+\frac{4k+37/12}{n\log n}\right).
\end{align*}
The claim follows when we do the following estimates: By inequality \eqref{eq:HLowerA1}, we can estimate $\log H \geq 7.182A(m,k)\geq 7.182A(3,1)$ and $n \geq 1865$.
\end{proof}

\begin{remark}
At the end of the fourth paragraph of the previous proof, we want to minimize the coefficients of the terms having terms $\log H$ multiplied by something which does not go to zero when $H$ or $n$ goes to infinity. Even more, we also want to keep the results relatively simple. Hence, in some of the cases we have used estimate \eqref{eq:estnLogn} once, twice or not at all depending on would it improve the coefficients or not.
\end{remark}

As a last part of this section, let us prove Corollary \ref{corollary:lowerBound}:

\begin{proof}[Proof of Corollary \ref{corollary:lowerBound}]
As in the previous proof, it is sufficient to estimate the term $\log |T(n+1,\mu)|$, where $n$ and $\mu$ are as in Lemma \ref{lemma:contra}.
By the proof of Theorem \ref{thm:lowerBound}, $\log |T(n+1,\mu)|$ is smaller than $\log H$ multiplied by
\begin{equation}
\label{eq:DMAx}
    <\frac{2k}{D}\log\log H-\frac{2k\log D}{D}+\frac{2k^2+3k\log 2-2k}{D}+\frac{4kA(m,k)}{D^2}+0.015k^2+ 0.235 k+1.034,
\end{equation}
where $A(m,k):=\left(2.539k\varphi(m)+5.724k+0.054\right)$ and $D:=D(m,k,\overline{\alpha}; R')$.
Further, we notice that $-\log D/D\leq 0$ if $D\geq 1$ and the function $-\log D/D$ is decreasing when $0<D<1$. By assumption \eqref{eq:BoundH}, Lemma \ref{lemma:xex} and since $A(m,k)>1$, inequality \eqref{eq:DMAx} is
\begin{equation*}
    <(\log\log H)^2\left(\frac{6k}{A(m,k)}+\frac{2k \log\log\log H}{A(m,k)\log\log H}+\frac{2k^2+3k\log 2-2k}{A(m,k)\log\log H}+\frac{0.015k^2+ 0.235 k+1.034}{(\log\log H)^2}\right).
\end{equation*}
Again, by \eqref{eq:HLowerA1} we can deduce that $\log H \geq 7.182A(3,1)$. Further, the term $A(m,k)$ in the denominators is at least $(2.539\cdot2 + 5.724)k$ and the claim follows. 
\end{proof}

\section{Preliminaries and proof of Theorem \ref{thm:lowerBoundMore}}
\label{sec:lower2}
In this section, we prove Theorem \ref{thm:lowerBoundMore}, i.e. that if we consider $(k+\varepsilon)\varphi(m)/(k+1)$ different residue classes, then we find a $p$-adic lower bound for a certain linear form. The proof follows similarly as in Section \ref{sec:FirstLower}: First we prove a contradiction that inequality \eqref{eq:Product} must be smaller than one under certain assumptions, then estimate such number $n$ and finally, prove Theorem \ref{thm:lowerBoundMore}.

\subsection{Contradiction and estimates for the number $n$}
\label{sec:Contra2}

Again, we would like to show a contradiction to estimate \eqref{eq:Product} and to prove simpler estimates for the number $n$ appearing in Lemma \ref{lemma:contra2}. 

First, let us define
\begin{align*}
    N_2(a)&:=-\varepsilon a\log a+(k+1)a\log\log a+a\left(\log c_1-k+\log k+2(k+1)(\left(\log m+6.55\right)\varphi(m)+1)\right)\\
   &\quad+\left(2.539+\frac{5.440}{\varphi(m)}\right)\frac{2\varphi(m)a(k+1)}{\log a} +\frac{2\log^2 a}{\log\log a}+(k+3.5)\log a+\frac{\log k\cdot\log a}{\log\log a}+k\log\log a \\
   &\quad+\log(k+1)+2.5\log k+ (k-1)\log\left(\max_{1\leq j \leq k}\{|\alpha_j|\}\right)+2\left(\left(\log m+6.55\right)\varphi(m)+1\right)k \\
   &\quad+\log\sqrt{2\pi}+1+\log H+\left(2.539+\frac{5.440}{\varphi(m)}\right)\frac{2\varphi(m)k}{\log a}+\frac{\log a}{a\log\log a}+\frac{29}{12a}.
\end{align*}
Then we can prove the contradiction:

\begin{lemma}
\label{lemma:contra2}
Assume that $m, k$ and $\lambda_0,\ldots, \lambda_k$ satisfy the same hypothesis as in Theorem \ref{thm:dmGene} and let $\alpha_1,\ldots, \alpha_k$ and $R$ be defined as in Theorem \ref{thm:lowerBoundMore}. Further, let also $\varepsilon \in (0,1)$ and $H$ be real numbers such that $H \geq \max_{0 \leq i \leq k}\{|\lambda_i|\}$ and 
$
\log H/\varepsilon \geq se^s, 
$
where $s \geq \max\{m^{\varphi(m)}+1,1866\}$. Even more, assume RH and that the $m$th cyclotomic Dedekind zeta function satisfies the GRH and define
\begin{equation}
\label{eq:nBoundEpsilon}
    n:=\max\left\{ a \in \mathbb{Z}: N_2(a) \geq 0\right\}
\end{equation}
and $\mu$ be as in \eqref{eq:defMu}.

Then, we have
\begin{equation*}
   (k+1)\max_{0 \leq i\leq k} \{\left|\lambda_i\right|\} \max_{0 \leq i \leq k} \{\left|B_{n+1,\mu,i}(1)\right|\}\cdot \max_{1 \leq i \leq k}\left\{ \prod_{p \in R\cap (\log(n+1),k(n+2)]}\left|S_{n+1,\mu,i}(1)\right|_p\right\}<1.
\end{equation*}
\end{lemma}

\begin{proof}
We approach the proof similarly as in the proof of Lemma \ref{lemma:contra}. First, we notice that from definition \eqref{eq:nBoundEpsilon} it follows that $\varepsilon (n+1) \log (n+1) > \log H$ and due to assumption \eqref{eq:HassumptionEpsilon} this means that $\log (n+1)> s > \max\{m^{\varphi(m)},1866\}$. Hence, by Lemma \ref{lemma:exists} there is at least one prime in $R\cap (\log(n+1),k(n+2)]$ and we can apply Lemma \ref{lemma:Log} and Remark \ref{remark:Log}. Thus, the logarithm of the left-hand side of the claim is 
\begin{align*}
    &<\log(k+1)+\log H+ (k-1)\log\left(\max_{1\leq j \leq k}\{|\alpha_j|\}\right)+(n+1)\log c_1+\log k+\log (n+1)+\frac{1}{n+1} \\
    &\quad+\left(k(n+1)+\mu+0.5\right)\log(k(n+1)+\mu)-k(n+1)-\mu+\log\sqrt{2\pi}+\frac{1}{12(k(n+1)+\mu)}\\ 
   &\quad+\frac{(k+\varepsilon)\varphi(m)}{k+1}\left( -\frac{(k(n+1)+\mu)\log (k(n+1)+\mu)}{\varphi(m)}+\left(2\left(\log m+6.55\right)+\frac{2}{\varphi(m)}\right)\cdot \right. \\
   &\quad\left.\cdot(k(n+1)+\mu) \right. \left.+\left(2.539+\frac{5.440}{\varphi(m)}\right)\frac{2(k(n+1)+\mu)}{\log(k(n+1)+\mu)}+\frac{(k(n+1)+\mu)\log\log (n+1)}{\varphi(m)}\right. \\
   &\quad \left.+\frac{\log (n+1)\cdot \log(k(n+1)+\mu)}{\varphi(m)\log\log (n+1)}+\frac{2\log (n+1)}{\varphi(m)} -\frac{(n+1)\log (n+1)}{\varphi(m)}+\left(2\left(\log m+6.55\right) \right.\right. \\
   &\quad \left.\left. +\frac{2}{\varphi(m)}\right)(n+1)+\left(2.539+\frac{5.440}{\varphi(m)}\right)\frac{2(n+1)}{\log (n+1)} +\frac{(n+1)\log\log (n+1)}{\varphi(m)}+\frac{\log^2 (n+1)}{\varphi(m)\log\log (n+1)} \right).
\end{align*}
Further, using inequalities $0 \leq \mu \leq k$, \eqref{eq:Logknk} and $k(n+1)+\mu \geq n+1$ for all $n$ under consideration and $0<\varepsilon<1$, the previous estimate is
\begin{align*}
   & <-\varepsilon (n+1)\log (n+1)+(k+1)(n+1)\log\log (n+1)+\left(2.539+\frac{5.440}{\varphi(m)}\right)\cdot \\
   &\quad\cdot \frac{2\varphi(m)(n+1)(k+1)}{\log (n+1)}+(n+1)\left(\log c_1-k+\log k+2(k+1)(\left(\log m+6.55\right)\varphi(m)+1)\right)\\
   &\quad+\frac{2\log^2 (n+1)}{\log\log (n+1)}+(k+3.5)\log (n+1)+\frac{\log k\cdot\log (n+1)}{\log\log (n+1)}+k\log\log (n+1) \\
   &\quad+\log(k+1)+2.5\log k+ (k-1)\log\left(\max_{1\leq j \leq k}\{|\alpha_j|\}\right)+2\left(\left(\log m+6.55\right)\varphi(m)+1\right)k \\
   &\quad+\log\sqrt{2\pi}+1+\log H+\left(2.539+\frac{5.440}{\varphi(m)}\right)\frac{2\varphi(m)k}{\log (n+1)}+\frac{\log (n+1)}{(n+1)\log\log (n+1)}+\frac{29}{12(n+1)} \\
   &\quad= N_2(n+1).
\end{align*}   
Because of the definition \eqref{eq:nBoundEpsilon}, the previous estimate is smaller than zero and the claim follows.
\end{proof}

Again, we would like to provide a simpler estimate for the number $n$. Hence, we first introduce the following lemma:
\begin{lemma}
\cite[Lemma 3.3]{EHMS2019M} 
\label{lemma:Inversez}
If $r \geq e$ is a real number, $z(y)$ is the inverse of the function $y(z)=z\log z$ and $y\geq re^r$, then
$$
z(y) \leq \left(1+\frac{\log r}{r}\right) \frac{y}{\log y}.
$$
\end{lemma}

Let us prove estimates for $n$:
\begin{lemma}
\label{lemma:nUpperEpsilon}
Let us use the same definitions and assumptions as in Lemma \ref{lemma:contra2}. Further, let also assume that $H$ satisfies the same assumptions as in Theorem \ref{thm:lowerBoundMore}. 

Then 
\begin{equation}
\label{eq:nEstUpperEpsilon}
    n\log n < \frac{15.195 k+16.195}{\varepsilon }n\log\log n+\frac{\log H}{\varepsilon }, \quad n< \frac{12.994\log H}{\varepsilon\log\left(\frac{\log H}{\varepsilon}\right)}, \quad n <0.072\log H,
\end{equation}
for $k \geq 1$ and
\begin{equation}
\label{eq:nEstUpperEpsilonv2}
    n< \frac{6.445\log H}{\varepsilon\log\left(\frac{\log H}{\varepsilon}\right)}, \quad n <0.020\log H 
\end{equation}
for $k \geq 2$.
\end{lemma}

\begin{proof}
By assumption \eqref{eq:nBoundEpsilon} we have $0 \leq N_2(n)$. We derive an upper bound
\begin{equation}
\label{eq:Amk3}
N_2(n) <    -\varepsilon n\log n+A(k)n\log\log n+\log H,
\end{equation}
where $A(k)$ depends only term $k$, for the term $N_2(n)$. Using this upper bound and keeping the inequality $0 \leq N_2(n)$ in mind, we find the wanted upper bounds.

Since by assumption \eqref{eq:nBoundEpsilon} we have $N_2(n+1) <0$, similarly as in the first paragraph in the proof of Lemma \ref{lemma:contra2} we can deduce that $\log(n+1)>s$. Hence, we have
$$
\log n > s-1 \geq \max\left\{1865,m^{\varphi(m)}, c_1\right\},
$$
and we also have $\log (k+1) <k$. Thus the term $N_2(n)$ is
\begin{align*}
   &N_2(n)< -\varepsilon n\log n+n\log\log n\left(k+2+\frac{2(k+1)\left(\log m+6.55\right)}{\log m}+\frac{2(k+1)}{\log\log n}\right.\\
   &\quad\left.+\left(2.539+\frac{5.440}{\varphi(m)}\right)\frac{2(k+1)}{\log n\cdot\log m}+\frac{2\log^2 n}{n\left(\log\log n\right)^2}+\frac{(k+3.5)\log n}{n\log\log n}+\frac{k\log n}{n\left(\log\log n\right)^2}+\frac{k+1}{n}\right. \\
   &\quad\left.+\frac{2\left(\log m+6.55\right)k}{n\log m}+\frac{5.5k+\log\sqrt{2\pi}+1}{n\log\log n}+\left(2.539+\frac{5.440}{\varphi(m)}\right)\frac{2k}{n\log n\cdot\log m} \right. \\
   &\quad\left.+\frac{29}{12n^2\log\log n} +\frac{\log n}{\left(n\log\log n\right)^2}\right)+\log H.
\end{align*}
We can set $\varphi(m)=2$, $\log n =1865$ and $\log m=3$.
Since $0 \leq N_2(n)$ and upper bound \eqref{eq:Amk3} with $A(k)=15.195 k+16.195$ holds, we have proved the first bound in \eqref{eq:nEstUpperEpsilon}. 

Let us now prove the second bound in \eqref{eq:nEstUpperEpsilon} and the first bound in \eqref{eq:nEstUpperEpsilonv2}. Because of the previous estimates, we also have
\begin{equation}
\label{eq:A2lower}
\varepsilon n\log n\left(1-A(k)\frac{\log\log n}{\varepsilon\log n}\right) \leq \log H.
\end{equation}
Further, since $\log\log n/\log n$ is decreasing for all $\log n \geq e$ and by assumption \eqref{eq:LowerBoundsForS} we have
$$
\log n > \left(\frac{A(k)}{\varepsilon}\right)^{1.5}> 31.39^{1.5},
$$
we can deduce that
$$
1-A(k)\frac{\log\log n}{\varepsilon\log n}>1-\frac{1.5\log\frac{A(k)}{\varepsilon}}{\left(A(k)/\varepsilon\right)^{0.5}}>0.
$$
Thus we have obtained
\begin{equation}
\label{eq:Lefty}
\frac{\left(A(k)/\varepsilon\right)^{0.5}}{\left(A(k)/\varepsilon\right)^{0.5}-1.5\log(A(k)/\varepsilon)}\cdot\log H\geq \varepsilon n\log n \geq \varepsilon re^{r},
\end{equation}
where 
\begin{equation*}
    r=\max\left\{1865, \left(\frac{15.195 k+16.195}{\varepsilon}\right)^{1.5}\right\}
\end{equation*}
by assumptions \eqref{eq:LowerBoundsForS}.
Let now $z$ be as in Lemma \ref{lemma:Inversez}. Hence, by applying Lemma \ref{lemma:Inversez}, setting $y$ be the left-hand side of inequality \eqref{eq:Lefty} divided by $\varepsilon$ and keeping in mind that the term $z(y)$ in Lemma \ref{lemma:Inversez} is increasing, we get
\begin{multline}
\label{eq:nlogelogy}
    n = z\left(n \log n \right)\leq z\left(\frac{\left(A(k)/\varepsilon\right)^{0.5}}{\left(A(k)/\varepsilon\right)^{0.5}-1.5\log(A(k)/\varepsilon)}\cdot\frac{\log H}{\varepsilon}\right) \\
    < \left(1+\frac{\log r}{r}\right)\cdot\frac{\left(A(k)/\varepsilon\right)^{0.5}}{\left(A(k)/\varepsilon\right)^{0.5}-1.5\log(A(k)/\varepsilon)}\cdot\frac{\log H/\varepsilon}{\log\left((\log H)/\varepsilon\right)}. 
\end{multline}
Next we prove the last two estimates in \eqref{eq:nEstUpperEpsilon} and estimates in \eqref{eq:nEstUpperEpsilonv2}.

Let us now prove the second estimate in \eqref{eq:nEstUpperEpsilon} and the first estimate in \eqref{eq:nEstUpperEpsilonv2}. We recognise the followings: The functions $\log r/r$ and $x^{0.5}/(x^{0.5}-1.5\log x)$ are decreasing and positive for all values $r, x$ under consideration. Hence, we can use $r=1865$ and estimate $A(k)/\varepsilon>A(1)$ or $>A(2)$ when $k \geq 2$ in the second term of the product in the last line of estimate \eqref{eq:nlogelogy}. These replacements prove the second estimate in \eqref{eq:nEstUpperEpsilon} and the first estimate in \eqref{eq:nEstUpperEpsilonv2}.

Let us now prove the last estimates in \eqref{eq:nEstUpperEpsilon} and \eqref{eq:nEstUpperEpsilonv2}. Using the assumption that the term $\log H/\varepsilon$ is greater than
$
\left(A(k)/\varepsilon\right)^{1.5}e^{\left(A(k)/\varepsilon\right)^{1.5}},
$
the right-hand side of estimate \eqref{eq:nlogelogy} is
\begin{equation*}
    <\frac{\left(A(k)/\varepsilon\right)^{1.5} \cdot\left(1+\frac{\log r}{r}\right)\cdot\frac{\log H}{A(k)}}{\left(\left(A(k)/\varepsilon\right)^{0.5}-1.5\log(A(k)/\varepsilon)\right)\left(\left(A(k)/\varepsilon\right)^{1.5}+1.5\log(A(k)/\varepsilon)\right)}.
\end{equation*}
Since the function depending on $A(k)/\varepsilon$ is decreasing for all possible values of $A(k)/\varepsilon$, we can again use the estimate $A(k)/\varepsilon>A(1)$ or $>A(2)$ when $k \geq 2$. Using also $r=1865$, we get the wanted estimates.
\end{proof}

\begin{remark}
For the left-hand side in \eqref{eq:A2lower} to be positive, it is enough to assume that $s \geq -\frac{A(k)}{\varepsilon}W_{-1}\left(-\frac{\varepsilon}{A(k)}\right)+1$ instead of the last assumption in \eqref{eq:LowerBoundsForS}. However, this would lead to more complicated computations. Hence, in order to keep the computations relatively simple, we have done a more lenient and a simpler assumption in \eqref{eq:LowerBoundsForS}.
\end{remark}

\begin{remark}
The exponent $1.5$ of $(15.195 k+16.195)/\varepsilon$ in assumptions \eqref{eq:LowerBoundsForS} is selected in such a way that it is the smallest exponent $r$ with one decimal accuracy that 
$
rx\log x/x^r<1
$
for all $x \geq 31.39$. 
\end{remark}

\subsection{Proof of Theorem \ref{thm:lowerBoundMore}}
\label{sec:proof4}

Let $n$ and $\mu$ be as in Lemma \ref{lemma:contra2}. Because of Lemmas \ref{lemma:exists}, \ref{lemma:contra2} and estimate \eqref{eq:Product} we must have 
$$
\left|B_{n+1,\mu,0}(1)\Lambda_p\right|_p > \left|\sum_{j=1}^k \lambda_jS_{n+1,\mu,j}(1)\right|_p 
$$
for some $p \in R\cap (\log(n+1),k(n+2)]$. We consider this prime number $p$. First we prove that $p$ is in set \eqref{eq:intervalEpsilon} and then that the wanted lower bound is satisfied.

Due to the second estimate in \eqref{eq:nEstUpperEpsilon}, number $p$ is smaller than the end of the interval in \eqref{eq:intervalEpsilon}. Even more, if  
\begin{equation}
\label{eq:wronglower}
    n+1 \leq \frac{\log H}{\varepsilon\log\left(\frac{\log H}{\varepsilon}\right)}< \frac{\log H}{\varepsilon\log\log H},
\end{equation}
then
\begin{equation}
\label{eq:lowern1}
    \log(n+1)(n+1) < \frac{\log H}{\varepsilon\log\log H}\log\left(\frac{\log H}{\varepsilon\log\log H}\right)=\frac{\log H}{\varepsilon}\left(1-\frac{\log\left(\varepsilon\log\log H\right)}{\log\log H}\right).
\end{equation}
Further, by assumptions \eqref{eq:HassumptionEpsilon} and \eqref{eq:LowerBoundsForS} and since $\varepsilon \in (0,1)$ we have
\begin{equation*}
    \log\log H>\varepsilon\log \log H>\varepsilon \log\left(\varepsilon/\varepsilon^{1.5}e^{\varepsilon^{-1.5}}\right)>\varepsilon^{-0.5}>1.
\end{equation*}
Hence, the right-hand side of estimate \eqref{eq:lowern1} is smaller than $\log H/\varepsilon$ if inequality \eqref{eq:wronglower} holds. However, this is a contradiction with the first paragraph of the proof of Lemma \ref{lemma:contra2}. Hence, inequality \eqref{eq:wronglower} cannot hold and hence $\log(n+1)$ is greater than the lower bound of interval in \eqref{eq:intervalEpsilon}. Thus $p$ is in the wanted set.

As in the proof of Theorem \ref{thm:lowerBound}, we only have to find a lower bound for the term $\left|\Lambda_p\right|_p$. By estimate \eqref{eq:TnEnough} it is sufficient to find an upper bound for the term $\left|T(n+1, \mu)\right|$.

Similarly as in the proof of Theorem \ref{thm:lowerBound}, we get
\begin{multline*}
    \log \left|T(n+1,\mu)\right| <\log H+(k+2.5)\log k+\frac{1}{k}+\log\sqrt{2\pi}+\frac{1}{12k(n+1)}-k+(k+1.5)\left(\log n+\frac{2}{n}\right) \\
    +\log \left(\max_{1 \leq j \leq k} \left\{|\alpha_j|\right\}\right)\left(kn+2k-1\right)
     +k(n+1)\left(\log 2+\log k+\log n+\frac{2}{n}-1\right).
\end{multline*}
Now, applying Lemma \ref{lemma:nUpperEpsilon} for $n \log n$ and $n$ and using assumptions \eqref{eq:LowerBoundsForS}, $k \geq 1$ and the facts that $\log k <k$, $\log H/\varepsilon>\log H$ and $\log(\log(0.020\log H))>0$, the previous estimate is
\begin{align*}
    &< \left(\frac{k}{\varepsilon}+1\right)\log H+n\log\log n\left(\frac{(15.195k + 16.195)k}{\varepsilon}+\frac{\log \left(\max_{1 \leq j \leq k} \left\{|\alpha_j|\right\}\right)k}{\log\log n}\right.\\
   &\quad\left.+\frac{k\left(\log k+\log 2-1\right)}{\log\log n}+\frac{(2k+1.5)\log n}{n\log\log n}+\frac{1}{kn\log\log n}+\frac{4k+37/12}{n^2\log\log n}\right. \\
   &\quad\left.+ \frac{(2k+2.5)\log k+\log\sqrt{2\pi}+k\log 2+(2k-1)\log \left(\max_{1 \leq j \leq k} \left\{|\alpha_j|\right\}\right)}{n\log\log n}\right) \\
   &\quad<\left(\frac{k}{\varepsilon}+1\right)\log H+\frac{a\log H}{\varepsilon\log\log H}\cdot\log\log \left(b\log H\right)\left(\frac{\left(15.195k + 16.195\right)k}{\varepsilon}+1+\frac{2}{n}\right.\\
   &\quad\quad\left.+\frac{k\left(k+\log 2-1\right)}{\log\log n}+\frac{(2k+1.5)\log n}{n\log\log n}+\frac{(2k+2.5+\log 2)k+\log\sqrt{2\pi}}{n\log\log n}+\frac{4k+37/12}{12n^2\log\log n} \right),
\end{align*}
where we can use $(a, b)=(12.994, 0.072)$ for all $k \geq 1$ and $(a, b)=(6.445, 0.020)$ for all $k \geq 2$.
Using $\log n \geq 1865$, we get the results.

\section*{Acknowledgements}

I would like to thank Dr. Louna Sepp\"al\"a for her support during the research process and Associate Professor Tim Trudgian for giving an idea to use Lemma \ref{lemma:Var} and Corollary \ref{corollary:varLower}. I also would like to thank an anonymous referee for many useful comments improving the manuscript.

\bibliographystyle{abbrv}
\bibliography{ref}

\end{document}